\documentclass[reqno]{amsart}
\usepackage{enumerate}
\newtheorem{theorem}{Theorem}[section]
\newtheorem{lemma}[theorem]{Lemma}
\newtheorem{proposition}[theorem]{Proposition}

\theoremstyle{definition}
\newtheorem{definition}[theorem]{Definition}

\newtheorem{remark}[theorem]{Remark}

\newcommand{\norm}[1]{\left\Vert#1\right\Vert}
\newcommand{\tvh}[1]{\left<#1\right>}

\newcommand{\tap}[1]{\left\{#1\right\}}

\numberwithin{equation}{section}
\usepackage{amsfonts}
\usepackage{multicol}

\usepackage{amsmath}
\usepackage{amssymb}
\usepackage{cases}
\usepackage{hyperref}
\usepackage{float}
\hypersetup{
   colorlinks,
    citecolor=blue,
    filecolor=black,
    linkcolor=blue,
    urlcolor=magenta
}
\hypersetup{linktocpage}

\begin{document}
\font\nho=cmr10
\def\dive{\operatorname{div}}
\def\cal{\mathcal}
\def\L{\cal L}

\def \ud{\underline }
\def\id{{\indent }}
\def\f{\frac}
\def\non{{\noindent}}
 \def\le{\leqslant} 
 \def\leq{\leqslant} 
\def\rar{\rightarrow}
\def\Rar{\Rightarrow}
\def\ti{\times}
\def\i{\mathbb I}
\def\j{\mathbb J}
\def\si{\sigma}
\def\Ga{\Gamma}
\def\ga{\gamma}
\def\ld{{\lambda}}
\def\Si{\Psi}
\def\f{\mathbf F}
\def\r{\hro{R}}
\def\e{\cal{E}}
\def\B{\cal B}
\def\A{\mathcal{A}}
\def\p{\mathbb P}

\def\tet{\theta}
\def\Tet{\Theta}
\def\hro{\mathbb}
\def\ho{\mathcal}
\def\P{\ho P}
\def\E{\mathcal{E}}
\def\n{\mathbb{N}}
\def\M{\mathbb{M}}
\def\dMu{\mathbf{U}}
\def\dMcs{\mathbf{C}}
\def\dMcu{\mathbf{C^u}}
\def\vk{\vskip 0.2cm}
\def\td{\Leftrightarrow}
\def\df{\frac}
\def\Wei{\mathrm{We}}
\def\Rey{\mathrm{Re}}
\def\s{\mathbb S}
\def\l{\mathcal{L}}
\def\C+{C_+([t_0,\infty))}
\def\o{\cal O}

\title[Periodic Solutions of Keller-Segel System]{Periodic Solutions of the parabolic-elliptic Keller-Segel system on whole spaces}
\author[P.T. Xuan]{Pham Truong Xuan*}
\address{Pham Truong Xuan \hfill\break
(*Corresponding author) Faculty of Mathematics and Informatics, Hanoi University of Science and Technology, 1 Dai Co Viet, Hanoi, Vietnam} 
\email{phamtruongxuan.k5@gmail.com} 

\author[T.V. Thuy]{ Tran Van Thuy}
\address{ Tran Van Thuy\hfill\break
East Asia University of Technology, 
Trinh Van Bo street, Nam Tu Liem, Hanoi, Vietnam}
\email{thuytv@eaut.edu.vn}

\author[N.T.V. Anh]{Nguyen Thi Van Anh}
\address{Nguyen Thi Van Anh\hfill\break
	Department of Mathematics, Hanoi National University of Education, 136 Xuan Thuy, Cau Giay, Hanoi, Vietnam}
\email{anhntv.ktt@hnue.edu.vn}

\author[N.T. Loan]{Nguyen Thi Loan}
\address{Nguyen Thi Loan\hfill\break Faculty of Fundamental Sciences, Phenikaa University, Hanoi 12116 Vietnam}
\email{loan.nguyenthi2@phenikaa-uni.edu.vn}

\begin{abstract}
In this paper, we investigate to the existence and uniqueness of periodic solutions for the parabolic-elliptic Keller-Segel system on whole spaces detailized by Euclidean space $\mathbb{R}^n\,\,(\hbox{ where   }n \geqslant 4)$ and real hyperbolic space $\mathbb{H}^n\,\, (\hbox{where   }n \geqslant 2)$. We work in framework of scritical spaces such as on weak-Lorentz space $L^{\frac{n}{2},\infty}(\mathbb{R}^n)$ to obtain the results for Keller-Segel system on $\mathbb{R}^n$ and on $L^{\frac{p}{2}}(\mathbb{H}^n)$ for  $n<p<2n$ to obtain the ones on $\mathbb{H}^n$. Our method is based on the dispersive and smoothing estimates of the heat semigroup and fixed point arguments. This work provides also a fully comparison between the asymptotic behaviours of periodic mild solutions of Keller-Segel system obtained in $\mathbb{R}^n$ and the one in $\mathbb{H}^n$. 
\end{abstract}



\maketitle

\tableofcontents

\section{Introduction}
Inspired by  the  pioneering work of Keller and Segel (see \cite{KeSe}), an extensive mathematical literature has grown on the Keller–Segel model based on biological phenomena describing
chemotaxis. The wide literature dealing with can be found in the exellence survey in \cite{Mu2002}.
To precisely state the problem under consideration, we are concerned with the parabolic-elliptic Keller-Segel (or Patlak-Keller-Segel) system (below, we write shortly that Keller-Segel (P-E) systems) on whole space $\mathcal{M}$, where $\mathcal{M}$ can be whole Euclidean space $\mathbb{R}^n\,\, (n \geqslant 4)$, or whole real hyperbolic space $\mathbb{H}^n \,\, (n \geqslant 2)$, and on the half-line time axis $\r_+$:
\begin{equation}\label{KS} 
\left\{
  \begin{array}{rll}
u_t \!\! &= \Delta_{\mathcal{M}}u - \chi\nabla_x \cdot (u\nabla_x v) + {F(t)} \quad  &x\in  \mathcal{M},\, t>0, \hfill \cr
-\Delta_{\mathcal{M}} v + \gamma v \!\!&=\; \kappa u \quad &x\in \mathcal{M},\,  t\geqslant 0, \cr
u(0,x) \!\!& = \;u_0(x) \quad & x \in \mathcal{M},
\end{array}\right.
\end{equation}
where the operator $\Delta_{\mathcal{M}}$ means Laplace-Beltrami operator on $\mathcal{M}$, which depends on the metric $g$ defined on $\mathcal{M}$. The unknown functions $u(t,x): \mathbb{R}_+\times \mathcal{M} \to \mathbb{R}_+$ represents the density of cells and $v(t,x): \mathbb{R}_+\times \mathcal{M} \to \mathbb{R}_+$ is the concentration of the chemoattractant. The parameter $\chi$ is the sensitivity parameter, which is a positive constant. The parameters $\gamma \geqslant 0$ and $\kappa >0$ denote the decay and production rate of the chemoattractant, respectively. The function ${F(t)}$ is given and we assume that ${F(t)}$ is time-periodic (see also \cite{Hi2020, Tan2021}).

In the study of modeling chemotaxis, the Keller-Segel (P-E) system \eqref{KS} and their modified versions are systems of nonlinear partial differential equations which describe chemical substances impact the movement of mobile species. Chemotaxis in microbiology refers to the migration of cells toward attractant chemicals or away from repellents. It represents the directed motion of cells in response to various chemical clues. Playing a key role in developmental biology, and more generally in self-organization of cell populations, numerous authors have contributed to strengthening the theory and applicability of Keller-Segel systems, see \cite{CoEs,Fe2011,Fe2021,Hi2020,MaPi,Mu2002,Tan2021} and the cited references therein. 

We notice that for the case $\gamma=0$ and $F(t)=0$, system \eqref{KS} has the scaling $u(t,x)\to \lambda^2u(\lambda^2t,\lambda x)$, which for initial data induces $u_0(x)\to \lambda^2u_0(\lambda x)$. If a space $X$ invariants under the scaling $u_0(x)\to \lambda^2u_0(\lambda x)$, i.e., $\norm{u_0}_X \simeq \norm{\lambda^2u_0(\lambda x)}_X$ for all $\lambda>0$, then $X$ is called critical space for system \eqref{KS}.

Concerning the case $\mathbb{R}^2$, the authors in \cite{Bla} proved that there exists a threshold value for the initial mass $M = \int u_0dx$ that relates to the existence and blow up of solutions: if $M<8\pi/\kappa$, then solutions exist globally  and if $M>8\pi/\kappa$, then solutions blow up in a finite time. For a bounded domain with smooth boundary in $\mathbb{R}^2$, Li and Wang established the finite-time blow-up and boundedness for system \eqref{KS} in \cite{Li2023}. On the other hand for the case $\mathbb{H}^2$, Pierfelice and Maheux obtained the local and global well-posedness results under the sub-critical condition and a blow-up result in \cite{MaPi}.

On Euclidean space $\mathbb{R}^n\, (n \geq 3)$, we recall some results about the existence and uniqueness of system \eqref{KS} with the initial data in critical spaces. {In \cite{Ko2}, Kozono and Sugiyama proved well-posedness of mild solutions for \eqref{KS} in weak-Lorentz spacces. Precisely, for small enough initial data $u_0\in L^{n/2, \infty}(\mathbb{R}^n)\, (n \geq 3)$, they proved that system \eqref{KS} has a unique mild solution $u\in C_b(\r_+, L^{n/2,\infty}(\mathbb{R}^n))$ satisfying that $t^{1-\frac{n}{2q}} u \in C_b(\r_+,\, L^q(\r^n))$, where $n/2<q<n$. Their method is based on a fixed  point argument by using Kato’s approach (see \cite{Ka1984}). After that, Kozono et al. \cite{Ko1} proved the existence of weak solutions $u\in C_b(\r_+,\, L^{n/2}(\mathbb{R}^n)) \cap L^q(\r_+,\, L^p(\mathbb{R}^n))$ with the initial data $u_0 \in L^{n/2}(\mathbb{R}^n)$ small enough, where $n\geq 3$ and $p,q$ are chosen suitable.} For the uniqueness in a large space of initial data, Ferreira \cite{Fe2021} established the unconditional uniqueness of mild solutions for system \eqref{KS} in $L^{p,\infty}(\mathbb{R}^n)$ spaces for $p>n/2$. In additionally, for other critical spaces we can refer some works for existence and uniqueness of mild solutions of (parabolic-parabolic or parabolic-eliptic) Keller-Segel systems in \cite{Biler,Chen2018,Fe2011,Iwa2011}.   
 
Concerning the periodic problem of chemotaxis-fluid models, some systematic-theoretic issues of the time-periodic for chemotaxis systems have been treated in \cite{Jin2020},\cite{Hi2020}, \cite{Negreanu2020},  \cite{Shi2023}, \cite{Tan2021} and \cite{Qi2017}. However, there is not a work which treats in detail about the existence of periodic solutions and their asymptotic behaviours for Keller-Segel (P-E) systems \eqref{KS} on whole spaces $\mathbb R^n$ or $\mathbb {H}^n$ in framework of critical spaces. This is the main motivation for our work in the present paper. 
We describle our method in this paper as follows: first, we prove the existence of bounded mild solutions of inhomogenous linear system corresponding to \eqref{KS}; then, we prove a Massera-type theorem by combining the ergodic approach, and the limit of Cesàro sum to construct an initial data which guarantees the existence of periodic mild solution of the linear system (see Theorem \ref{main1} in Subsection \ref{Su23}); then, we prove the existence and uniqueness of small periodic mild solutions of system \eqref{KS} by using the obtained results  for the linear system and fixed point arguments (see Theorem \ref{main2} and Theorem \ref{main3}). In addition, we establish the asymptotic behaviour of periodic mild solutions in Euclidean space $\mathbb{R}^n \, (n \geqslant 4)$ (see Theorem \ref{main2}) by using dispersive estimates and fixed point arguments and in real hyperbolic space $\mathbb{H}^n\, (n \geqslant 2)$ (see Theorem \ref{main3}) by using Gronwall's inequality. The asymptotic behaviour shows the main difference between the two cases: in detail, we prove that the periodic mild solutions are polynomial stable in $\mathbb{R}^n$, meanwhile, they decay exponentially (hence, they are exponential stable) in $\mathbb{H}^n$. Moreover, we prove a unconditional uniqueness theorem for bounded mild solutions in the hyperbolic case (see Theorem \ref{main4} in Subsection \ref{Su33}) by developing methods in \cite{Fe2021,Xuan2022}.

This paper is structured as follows. In Section \ref{S2}, we recall the notation of  Lorentz spaces and study the periodic solution to \eqref{KS} in whole space $\mathbb R^n$, where $n \geqslant 4$. The periodic solution to \eqref{KS} in real hyperbolic space $\mathbb{H}^n\, (n\geqslant 2)$ is dealt in Section \ref{S3}. The paper ends with some concluding remarks in Section \ref{S4}.\\
{\bf Notations.}
{Let $X$ be a Banach space, we denote the space of all continuous functions from $\mathbb{R}_+$ to $X$ (where $\mathbb{R}_+ = (0,\infty)$) by $C_b(\mathbb{R}_+,X)$ which is a Banach space equipped with the norm
$$\|v\|_{\infty,X}:=\sup_{t\in\r_+}\|v(t)\|_X.$$}

\section{Periodic solution for Keller-Segel (P-E) system on $\mathbb{R}^n$}\label{S2}
\subsection{Lorentz spaces}
In this section, we recall some facts about Lorentz spaces. For further details,
the reader is referred to \cite{Ba1996,BeLo,Hu,One}.
\begin{definition}
Let $\Omega$ be a {measurable} subset of $\mathbb{R}^n$, a measurable function $f$ defined on $\Omega$
belongs to the Lorentz space $L^{p,q}(\Omega)$ if the quantity
\begin{subnumcases}{\norm{f}_{L^{p,q}} =}
\left( \frac{p}{q}\int_0^\infty\left[ t^{\frac{1}{p}}f^{**}(t)\right]^q\frac{dt}{t} \right)^{\frac{1}{q}} & $\hbox{if  } 1<p<\infty,\, 1\leq q <\infty,$ \cr
\sup_{t>0} t^{\frac{1}{p}}f^{**}(t) & $\hbox{if  } 1<p\leq \infty,\, q=\infty$\nonumber  
\end{subnumcases}
is finite, where
$$f^{**}(t) = \frac{1}{t}\int_0^tf^{*}(s)ds$$
and
$$f^*(t)=\inf \left\{s>0: m(\left\{ x\in \Omega:|f(x)|>s\right\})\leq t \right\},\, t> 0$$
with $m$ denoting the Lebesgue measure in $\mathbb{R}^n$.
\end{definition}

The space $L^{p,q}(\Omega)$ with the norm $\norm{f}_{L^{p,q}}$ is a Banach space. In particular, $L^p(\Omega)=L^{p,p}(\Omega)$ and $L^{p,\infty}$ is called weak-$L^p$ space or the Marcinkiewicz space.
For $1\leq q_1\leq p \leq q_2\leq \infty$ we have the following relation
\begin{equation}\label{Inclusion}
L^{p,1}(\Omega)\subset L^{p,q_1}(\Omega)\subset L^p(\Omega) \subset L^{p,q_2}(\Omega)\subset L^{p,\infty}(\Omega).
\end{equation}

Let $1<p_1,p_2,p_3\leq \infty$ and $1\leq r_1,r_2,r_3\leq \infty$ be such that $\dfrac{1}{p_3}=\dfrac{1}{p_1}+ \dfrac{1}{p_2}$ and $\dfrac{1}{r_1} + \dfrac{1}{r_2}\geqslant \dfrac{1}{r_3}$. We have the H\"older inequality (see \cite{Hu,One}):
\begin{equation}\label{Holder}
\norm{fg}_{L^{p_3,r_3}} \leq C\norm{f}_{L^{p_1,r_1}}\norm{g}_{L^{p_2,r_2}},
\end{equation}
where $C > 0$ is a constant independent of $f$ and $g$. We also have Young inequality in Lorentz spaces
defined on $\Omega=\mathbb{R}^n$. In particular, if $1<p_1,p_2,p_3\leq \infty$ and $1\leq r_1,r_2,r_3\leq \infty$ be such that $\dfrac{1}{p_3}=\dfrac{1}{p_1}+ \dfrac{1}{p_2}-1$ and $\dfrac{1}{r_1} + \dfrac{1}{r_2}\geqslant \dfrac{1}{r_3}$, then (see \cite{One}):
\begin{equation}
\norm{f*g}_{L^{p_3,r_3}} \leq C\norm{f}_{L^{p_1,r_1}}\norm{g}_{L^{p_2,r_2}}.
\end{equation}
Moreover, in the case $p_1=1$ and $1<p=p_2=p_3\leq \infty$, the following inequality holds (see \cite{Ba1996,BeBr}):
\begin{equation}\label{Young}
\norm{f*g}_{L^{p,\infty}} \leq \frac{p^{1+\frac{1}{p}}}{p-1}\norm{f}_{L^1}\norm{g}_{L^{p,\infty}}.
\end{equation}

We now recall an interpolation property of Lorentz spaces. For $0<p_1<p_2\leq \infty$, $0<\theta<1$, $\dfrac{1}{p} = \dfrac{1-\theta}{p_1}+ \dfrac{\theta}{p_2}$ and $1\leq r_1,r_2,r\leq\infty$, we have (see \cite[Theorems 5.3.1 and 5.3.2]{BeLo}):
\begin{equation}
\left( L^{p_1,r_1},L^{p_2,r_2}\right)_{\theta,r} = L^{p,r},
\end{equation}
where $L^{p_1,r_1}$ is endowed with $\norm{.}^*_{L^{p_1,r_1}}$ instead of $\norm{.}_{L^{p_1,r_1}}$ when $0<p_1\leq r_1$. {Here, we denoted (see \cite{Ba1996,Hu}):
\begin{subnumcases}{\norm{f}^*_{L^{p,q}} =}
\left( \frac{p}{q}\int_0^\infty\left[ t^{\frac{1}{p}}f^{*}(t)\right]^q\frac{dt}{t} \right)^{\frac{1}{q}} & $\hbox{if  } 0<p<\infty,\, 0< q <\infty,$ \cr
\sup_{t>0} t^{\frac{1}{p}}f^{*}(t) & $\hbox{if  } 0<p\leq \infty,\, q=\infty$\nonumber  
\end{subnumcases}
}
{Furthermore, if $1\leqslant q < \infty$, then we have the following duality
\begin{equation}
(L^{p,q})' = L^{p',q'},\hbox{   where   } p'=\frac{p}{p-1},\, q'=\frac{q}{q-1} \hbox{   and   } q'=\infty \hbox{   if   } q=1.
\end{equation}}
By interpolating \eqref{Young}, we get
\begin{equation}\label{InYoung}
\norm{f*g}_{L^{p,r}} \leq C \norm{f}_{L^1}\norm{g}_{L^{p,r}},
\end{equation}
for $1<p\leq \infty$ and $1\leq r \leq \infty$. In the case $1\leq p=r\leq \infty$, the inequality \eqref{InYoung} holds if the space $L^1$ is replaced by
the space of signed finite measures $\mathcal{M}$ endowed with the norm of the total variation 
$\norm{\mu}_{\mathcal{M}} = |\mu|(\mathbb{R}^n)$.

\subsection{Keller-Segel (P-E) system and some useful estimates on $\mathbb{R}^n$}
Let $F\in C_b(\mathbb{R}_+, L^{\frac{n}{3},\infty}(\mathbb{R}^n))$ and $u_0(x) \in L^{\frac{n}{2},\infty}(\mathbb{R}^n)$, we consider the Keller-Segel (P-E) system on the framework of weak-Lorentz space $L^{\frac{n}{2},\infty}(\mathbb{R}^n)$, where $n \geqslant 4$: 
\begin{equation}\label{KSRn} 
\left\{
  \begin{array}{rll}
u_t \!\! &= \Delta u - \chi\nabla_x \cdot (u\nabla_x v) + {F(t)} \quad  &x\in  \mathbb{R}^n,\, t>0, \hfill \cr
-\Delta v + \gamma v \!\!&=\; \kappa u \quad &x\in \mathbb{R}^n,\,  t\geqslant 0, \cr
u(0,x) \!\!& = \;u_0(x) \quad & x \in \mathbb{R}^n.
\end{array}\right.
\end{equation}
For the sake of convenience, we assume that the constants $\chi=1$ and denote the Laplace-Beltrami operator $\Delta_{\mathbb{R}^n}$ by $\Delta$. Moreover, for the simplicity of calculations we study the case ${F(t)}= \nabla_x \cdot f(t,\cdot)$. {Note that the condition $n\geqslant 4$ is arised from the bilinear estimates in Proposition \ref{usefulEs} below.}

The second equation of this system leads to $v= \kappa (-\Delta+ \gamma I)^{-1}u$. Therefore, according to Duhamel’s principle, the Cauchy problem \eqref{KSRn} can be formally converted to the integral equation
\begin{equation}\label{00inte}
u(t) = e^{t\Delta }u_0 + \int_0^t \nabla_x \cdot e^{(t-s)\Delta}\left[ -\kappa u\nabla_x(-\Delta_x+ \gamma I)^{-1}u + f \right](s)ds.
\end{equation}
{The above integral equation \eqref{00inte} should be mean in a dual sense in $L^{\frac{n}{2},\infty}(\mathbb{R}^n)$- setting, in other words,
in distributional sense, as in \cite{Ba1996,Bo1995,Ya2000}. More precisely, this means that
\begin{equation}\label{0inte}
\left<u(t),\phi\right> = \left<e^{t\Delta}u_0,\phi \right> + \left<\int_0^t \nabla_x \cdot e^{(t-s)\Delta}\left[- \kappa u\nabla_x(-\Delta + \gamma I)^{-1}u + f \right](s)ds,\phi\right>,
\end{equation}
for all $\phi\in L^{\frac{n}{2},1}(\mathbb{R}^n)$ and $t>0$.}

{Note that, for a given $\omega\in L^{\frac{n}{2},\infty}(\mathbb{R}^n)$, the function $t\mapsto e^{t\Delta}\omega$ is continuous for $t>0$ and is weakly-$^*$ continuous at $t=0$ (see for example \cite{Ba1996,Bo1995}). Therefore, we define the global mild solution of system \eqref{KSRn} as follows:
\begin{definition}\label{De1}
Let $u_0 \in L^{\frac{n}{2},\infty}(\mathbb{R}^n)$. A function $u(t,x)$ satisfying
\begin{equation}\label{0limit}
\lim_{t\to 0^+} \left<u(t),\phi\right> = \left<u_0,\phi \right> \hbox{   for all   } \phi \in L^{\frac{n}{2},1}(\mathbb{R}^n)
\end{equation}
is said a global mild solution for the initial value problem \eqref{KSRn} if $u(t,x)$ satisfies the integral equation \eqref{00inte} in the sense of duality \eqref{0inte}. 
\end{definition}}

Setting $L_j = \partial_j(-\Delta + \gamma I)^{-1}$, the properties of this operator is given in the following lemma (see \cite[Lemma 4.1]{Fe2021}):
\begin{lemma}\label{invertEs}
Let $\gamma\geqslant 0$, $n\geqslant 2$, $1<p<n$, $1\leqslant d \leqslant \infty$ and $\frac{1}{q}=\frac{1}{p}-\frac{1}{n}$.
The operator $L_j$ is continuous from $L^{p,d}(\mathbb{R}^n)$ to $L^{q,d}(\mathbb{R}^n)$, for each $j=1,\, 2\, {\cdots} \,n$. Moreover, there exists a constant $C>0$ independent of $f$ and $\gamma$ satisfying 
\begin{equation}
\norm{L_jf}_{L^{q,d}} \leqslant Cg(\gamma)\norm{f}_{L^{p,d}},
\end{equation}
where $g(0)=1$ and $g(\gamma)=\gamma^{-(n-1)}$ if $\gamma>0$.
\end{lemma}
We recall some results of heat semigroup in Lorentz spaces (see \cite{Ba1996,Ya2000}):
\begin{lemma}\label{YamaEs}
\begin{itemize}
\item[(i)]  (Dispersive and smoothing estimates {(see Barraza \cite{Ba1996})}) Let $m \in \left\{ 0 \right\}\cup \mathbb{N}$, $1<r \leqslant p \leqslant \infty$, and $1 \leqslant d_1\leqslant d_2 \leqslant \infty$. Then, there exists a constant $C > 0$ such that
\begin{equation}\label{smoothEs}
\norm{\nabla_x^m e^{\Delta t}\varphi}_{L^{p,d_2}} \leqslant C t^{-\frac{m}{2}-\frac{n}{2}\left( \frac{1}{r}-\frac{1}{p} \right)}\norm{\varphi}_{L^{r,d_1}}, \hbox{   for all   } \varphi \in L^{r,d_1}.
\end{equation}
\item[(ii)] (Yamazaki's estimate {(see Yamazaki \cite{Ya2000})}) Let $1<r<p<\infty$. There is a constant $C>0$ such that
\begin{equation}\label{YaEs}
\int_0^\infty s^{\frac{n}{2}\left( \frac{1}{r} -\frac{1}{p} \right)-\frac{1}{2}} \norm{\nabla_x\cdot e^{\Delta s}\phi}_{L^{p,1}}ds \leqslant C\norm{\phi}_{L^{r,1}}, \hbox{   for all   } \phi\in L^{r,1}.
\end{equation}
\end{itemize}
\end{lemma}
Setting the linear operator
\begin{equation}\label{LinearOP}
L(f)(t)= \int_0^t \nabla_x \cdot e^{\Delta s}f(s,\cdot)ds,
\end{equation}
and the bilinear operator
\begin{equation}\label{BilinearOP}
B(u,\omega)(t)= \int_0^t \nabla_x \cdot e^{(t-s)\Delta}\left[ u\nabla_x(-\Delta + \gamma I)^{-1}\omega \right](s)ds.
\end{equation}
Lemma \ref{invertEs} and Lemma \ref{YamaEs} lead to the boundedness of linear and bilinear operators $L(\cdot)$ and $B(\cdot,\cdot)$ (see \cite[Lemma 4.3 and Theorem 3.1 (i)]{Fe2021}). 
\begin{proposition}\label{usefulEs}
\begin{itemize}
\item[(i)] (Linear estimate) Let $n\geqslant 2$ and $1<r<p<\infty$ be such that $\frac{1}{r}-\frac{1}{p}=\frac{1}{n}$. There exists a positive constant ${\widetilde{C}}$ satisfying
\begin{equation}
\norm{L(f)(t)}_{L^{p,\infty}} \leqslant \widetilde{C}\sup_{t>0}\norm{f(t,\cdot)}_{L^{r,\infty}}.
\end{equation}
\item[(ii)] (Bilinear estimate) Let $n\geqslant 4$ and $\gamma\geqslant 0$. There exists a positive constant $K$ satisfying
\begin{equation}
\norm{B(u,\omega)}_{\infty,L^{\frac{n}{2},\infty}} \leqslant Kg(\gamma)\norm{u}_{\infty,L^{\frac{n}{2},\infty}}\norm{\omega}_{\infty, L^{\frac{n}{2},\infty}},
\end{equation}
for all $u,\, \omega$ belong to $C_b(\mathbb{R}_+; L^{\frac{n}{2},\infty}(\mathbb{R}^n))$, where $g(\gamma)=1$ if $\gamma=0$ and $g(\gamma)=\gamma^{-(n-1)}$ if $\gamma>0$.
\end{itemize}
\end{proposition}
{We notice that the condition of space's dimension $n\geqslant 4$ in Assertion (ii) in Lemma \ref{usefulEs} is arised from technical issues which guarantees dual estimates (see the proof of Theorem 3.1 in \cite{Fe2021}).}

\subsection{Bounded mild periodic solutions and their polynomial stability}\label{Su23}
We consider the inhomogeneous linear system corresponding to system \eqref{KSRn}:
\begin{equation}\label{LinearKSRn} 
\left\{
  \begin{array}{rll}
u_t \!\! &= \Delta u + \nabla_x \cdot \left[ -\kappa\omega\nabla_x (-\Delta + \gamma I)^{-1}\omega + f(t)\right] \quad  &x\in  \mathbb{R}^n,\, t>0, \hfill \cr
u(0,x) \!\!& = \;u_0(x) \quad & x \in \mathbb{R}^n
\end{array}\right.
\end{equation}
for a given function $\omega \in C_b(\mathbb{R}_+, L^{\frac{n}{2},\infty}(\mathbb{R}^n))$.
By Duhamel's principle, this system can be reformulated by the following integral equation
\begin{equation}\label{inte2}
u(t) = e^{t\Delta}u_0 + \int_0^t \nabla_x \cdot e^{(t-s)\Delta}\left[- \kappa \omega\nabla_x(-\Delta + \gamma I)^{-1}\omega + f \right](s)ds.
\end{equation}
{Similar to the linear case, the above integral equation \eqref{inte2} should be mean in a dual sense in $L^{\frac{n}{2},\infty}(\mathbb{R}^n)$- setting. This means that
\begin{equation}\label{inte2222}
\left<u(t),\phi\right> = \left<e^{t\Delta}u_0,\phi \right> + \left<\int_0^t \nabla_x \cdot e^{(t-s)\Delta}\left[- \kappa \omega\nabla_x(-\Delta + \gamma I)^{-1}\omega + f \right](s)ds,\phi\right>,
\end{equation}
for all $\phi\in L^{\frac{n}{2},1}(\mathbb{R}^n)$ and $t>0$.
Similar to Definition \ref{De1}, we have the following definition
\begin{definition}
Let $u_0 \in L^{\frac{n}{2},\infty}(\mathbb{R}^n)$. A function $u(t,x)$ satisfying
\begin{equation}\label{limit}
\lim_{t\to 0^+} \left<u(t),\phi\right> = \left<u_0,\phi \right> \hbox{   for all   } \phi \in L^{\frac{n}{2},1}(\mathbb{R}^n).
\end{equation}
is said a global mild solution for the initial value problem \eqref{LinearKSRn} if $u(t,x)$ satisfies the integral equation \eqref{inte2} in the sense of duality \eqref{inte2222}. 
\end{definition}
}

Using Lemma \ref{YamaEs} and Proposition \ref{usefulEs}, we now establish the existence of bounded periodic mild solution of system \eqref{LinearKSRn}.
\begin{theorem}\label{main1}
\begin{itemize}
\item[(i)] (Boundedness) For given functions $\omega \in C_b(\mathbb{R}_+, L^{\frac{n}{2},\infty}(\mathbb{R}^n))$ and $f\in C_b(\mathbb{R}_+, L^{\frac{n}{3},\infty}(\mathbb{R}^n))$, there exists a unique mild solution of system \eqref{LinearKSRn} satisfying the integral equation \eqref{inte2}. Moreover, the following boundedness holds
\begin{equation}\label{boundedness}
\norm{u(t)}_{L^{\frac{n}{2},\infty}(\mathbb{R}^n)} \leqslant C\norm{u_0}_{L^{\frac{n}{2},\infty}(\mathbb{R}^n)} + \kappa Kg(\gamma)\norm{\omega}^2_{\infty,L^{\frac{n}{2},\infty}(\mathbb{R}^n)} + \widetilde{C}\norm{f}_{\infty,L^{\frac{n}{3},\infty}(\mathbb{R}^n)}.
\end{equation}
\item[(ii)] (Massera-type principle) If $\omega$ and $f$ are $T$-periodic functions depending on the time variable, then {there exists a function $\hat{\xi}\in L^{\frac{n}{2},\infty}(\mathbb{R}^n)$ which guarantees that the linear system \eqref{LinearKSRn} with the initial data $\hat{\xi}$ has a unique bounded $T$-periodic mild solution $\hat{u}(t)$ satisfying the following boundedness}
\begin{equation}\label{boundednessPeriodic}
\norm{\hat{u}(t)}_{L^{\frac{n}{2},\infty}(\mathbb{R}^n)} \leqslant (C+1)\left( \kappa Kg(\gamma)\norm{\omega}^2_{\infty,L^{\frac{n}{2},\infty}(\mathbb{R}^n)} + \widetilde{C}\norm{f}_{\infty,L^{\frac{n}{3},\infty}(\mathbb{R}^n)}\right).
\end{equation}
\end{itemize}
\end{theorem}
\begin{proof}
(i) {Let $u(t,x)$ satisfies integral equation \eqref{inte2}, we prove that $u(t,x)$ is a  unique global mild solution of \eqref{LinearKSRn} by establishing the boundedness of $\norm{u(t)}_{L^{\frac{n}{2},\infty}(\mathbb{R}^n)}$ and the convergence \eqref{limit}.}
Using Assertion (i) in Lemma \ref{YamaEs} and Proposition \ref{usefulEs}, we have
\begin{eqnarray*}
\norm{u(t)}_{L^{\frac{n}{2},\infty}(\mathbb{R}^n)} &\leqslant& \norm{e^{t\Delta}u_0}_{L^{\frac{n}{2},\infty}(\mathbb{R}^n)} +  \norm{\int_0^t \nabla_x \cdot e^{(t-s)\Delta}\left[ -\kappa\omega\nabla_x(-\Delta + \gamma I)^{-1}\omega +f \right](s)ds}_{L^{\frac{n}{2},\infty}(\mathbb{R}^n)}\cr
&\leqslant& C\norm{u_0}_{L^{\frac{n}{2},\infty}(\mathbb{R}^n)} + \kappa\norm{B(\omega,\omega)(t)}_{{L^{\frac{n}{2},\infty}(\mathbb{R}^n)}} + \norm{L(f)}_{\infty,L^{\frac{n}{2}}(\mathbb{R}^n)}\cr
&\leqslant& C\norm{u_0}_{L^{\frac{n}{2},\infty}(\mathbb{R}^n)} + \kappa K g(\gamma)\norm{\omega}^2_{\infty,{L^{\frac{n}{2},\infty}(\mathbb{R}^n)}} + \widetilde{C}\norm{f}_{{\infty,L^{\frac{n}{3},\infty}(\mathbb{R}^n)}}.
\end{eqnarray*}
{The boundedness of $u$ holds. It remains to show the convergence \eqref{limit}:
\begin{equation*}
\lim_{t\to 0^+} \left<u(t),\phi\right> = \left<u_0,\phi \right> \hbox{   for all   } \phi \in L^{\frac{n}{2},1}(\mathbb{R}^n).
\end{equation*}
Indeed, we have
\begin{eqnarray}\label{limittozero}
\lim_{t\to 0^+}\left<u(t),\phi\right> &=&\lim_{t\to 0^+} \left<e^{t\Delta}u_0,\phi \right> + \lim_{t\to 0^+}\left<\int_0^t \nabla_x \cdot e^{(t-s)\Delta}\left[- \kappa \omega\nabla_x(-\Delta + \gamma I)^{-1}\omega + f \right](s)ds,\phi\right>\cr
&=& \lim_{t\to 0^+} \left<e^{t\Delta}u_0,\phi \right> + 0\cr
&=& \lim_{t\to 0^+} \left<u_0,e^{t\Delta}\phi \right> \cr
&=& \left<u_0,\phi \right> 
\end{eqnarray}
for all $\phi\in L^{\frac{n}{2},1}(\mathbb{R}^n)$. The last limit follows from the strong continuity at $0$ of the heat semigroup $e^{t\Delta}$ over Lorentz space $L^{\frac{n}{2},1}(\mathbb{R}^n)$ (see \cite{Bo1995,Fe2010}).} 
Therefore, the proof of Assertion (i) is completed.\\

(ii) We employ the method in \cite[Theorem 2.3]{HuyXuan2022} to {prove the existence of function $\hat{\xi} \in L^{\frac{n}{2},\infty}(\mathbb{R}^n)$ which guarantees that the system \eqref{LinearKSRn} with the initial data $\hat{\xi}$ has a unique periodic mild solution}. For each initial data $\xi\in L^{\frac{n}{2},\infty}(\mathbb{R}^n)$ from Assertion (i) it follows that there exists a unique mild solution $u\in C_b(\r_+, L^{\frac{n}{2},\infty}(\mathbb{R}^n)$ to 
system \eqref{LinearKSRn} with $u(0)=\xi$. 
This fact allows to define the Poincar\'e map $\P : L^{\frac{n}{2},\infty}(\mathbb{R}^n) \to L^{\frac{n}{2},\infty}(\mathbb{R}^n)$ as follows: For each $\xi \in L^{\frac{n}{2},\infty}(\mathbb{R}^n)$, we set
\begin{equation}
\begin{split}\label{Po}
\P(\xi)&:=u(T) \hbox{ where } u\in C_b(\r_+, L^{\frac{n}{2},\infty}(\mathbb{R}^n))\hbox{ is the unique mild solution of \eqref{LinearKSRn}}\hbox{  with }u(0)= \xi.  
\end{split}
\end{equation}
\def\A{\mathcal{A}}
Note that using the formula \eqref{inte2} of the solutions, we see that 
\begin{eqnarray}\label{Po1}
\P(\xi)=u(T)=e^{T\Delta}\xi +  \int_0^T \nabla_x \cdot e^{\Delta (T-s)}\left[ -\kappa\omega\nabla_x(-\Delta + \gamma I)^{-1}\omega + f\right](s)ds
\end{eqnarray}
with $u$  as in \eqref{Po}. 
Next, from $T$-periodicity of $\omega$ and $f$, we see that
\begin{eqnarray*}
 u((k+1)T)&=& e^{(k+1)T\Delta}\xi  +  \int\limits^{(k+1)T}_{0} \nabla_x \cdot e^{((k+1)T-s)\Delta }\left[-\kappa \omega\nabla_x(-\Delta + \gamma I)^{-1}\omega + f \right](s)ds \cr
&=& e^{(k+1)T\Delta}\xi  + \int\limits^{kT}_{0} \nabla_x \cdot e^{((k+1)T-s)\Delta }\left[ -\kappa \omega\nabla_x(-\Delta + \gamma I)^{-1}\omega + f\right](s)ds \cr
&&+ \int\limits^{(k+1)T}_{kT} \nabla_x \cdot e^{((k+1)T-s)\Delta }\left[ -\kappa\omega\nabla_x(-\Delta + \gamma I)^{-1}\omega + f\right](s)ds\cr
&=& e^{(k+1)T\Delta}\xi  + 
\int\limits^{kT}_{0}e^{T\Delta}\nabla_x \cdot e^{ (kT-s)\Delta}\left[- \kappa\omega\nabla_x(-\Delta + \gamma I)^{-1}\omega + f\right](s)ds \cr
&& +  \int\limits^{T}_{0}\nabla_x \cdot e^{(T-s)\Delta }\left[- \kappa\omega\nabla_x(-\Delta + \gamma I)^{-1}\omega + f\right](s)ds\cr
&=& e^{T\Delta}e^{kT\Delta}\xi +
\int\limits^{kT}_{0} e^{T\Delta}\nabla_x \cdot e^{(kT-s)\Delta }\left[ -\kappa\omega\nabla_x(-\Delta + \gamma I)^{-1}\omega + f\right](s)ds\cr
&&+ \int\limits^{T}_{0} \nabla_x \cdot e^{(T-s)\Delta }\left[ -\kappa \omega\nabla_x(-\Delta + \gamma I)^{-1}\omega + f\right](s)ds\cr
&=& e^{T\Delta}u(kT) + \int\limits^{T}_{0}\nabla_x \cdot e^{(T-s)\Delta}\left[ -\kappa \omega\nabla_x(-\Delta + \gamma I)^{-1}\omega + f\right](s)ds \hbox{ for all }k\in \n.
\end{eqnarray*}
It follows  that  
$\P^k(\xi)=u(kT)\hbox{ for all }k\in \n$. Thus,  $\{\P^k(\xi)\}_{k\in \n}$ is a  bounded sequence in  $L^{\frac{n}{2},\infty}(\mathbb{R}^n)$.

Next, as usual in ergodic theory, for each $n\in \n$ we define the Ces\`aro sume  $\mathbf{P}_n$ by
\begin{equation}\label{cesa}
\mathbf{P}_n:=\frac{1}{n}\sum_{k=1}^n\P^k : L^{\frac{n}{2},\infty}(\mathbb{R}^n)\to L^{\frac{n}{2},\infty}(\mathbb{R}^n).
\end{equation}
Starting from $\xi = 0$ and using the inequality \eqref{boundedness}, we obtain
\begin{equation}\label{boundp}
\sup_{k\in\n}\|\P^k(0)\|_{L^{\frac{n}{2},\infty}}\le \kappa K g(\gamma)\|\omega\|^2_{\infty,L^{\frac{n}{2},\infty}} + \widetilde{C}\norm{f}_{\infty, L^{\frac{n}{3},\infty}}.
\end{equation}
The boundedness of  $\{\P^k(0)\}_{k\in \n}$ in  
$L^{\frac{n}{2},\infty}(\mathbb{R}^n)$ implies  that the sequence
$\{\mathbf{P}_n(0)\}_{n\in \n}=\tap{\frac{1}{n}\sum_{k=1}^n\P^k(0)}_{n\in \n}$ is clearly a bounded sequence in $L^{\frac{n}{2},\infty}(\mathbb{R}^n)$. 

Since $L^{\frac{n}{2},\infty}(\mathbb{R}^n)$ has a separable pre-dual $L^{\frac{n}{n-2},1}(\mathbb{R}^n)$, from Banach-Alaoglu's Theorem there exists a subsequence $\{\mathbf{P}_{n_k}(0)\}$ of 
$\{\mathbf{P}_{n}(0)\}$ such that
\begin{equation}\label{suse}
\{\mathbf{P}_{n_k}(0)\}\ {\xrightarrow{weak\hbox{-}^*}}{}\  \hat{\xi}\in L^{\frac{n}{2},\infty}(\mathbb{R}^n) \hbox{ with } \|\hat{\xi}\|_{L^{\frac{n}{2},\infty}}\le \kappa K g(\gamma)\|\omega\|^2_{\infty,L^{\frac{n}{2},\infty}} + \widetilde{C}\norm{f}_{\infty, L^{\frac{n}{3},\infty}}.
\end{equation}
By simple computations using formula \eqref{cesa} we obtain
$\P\mathbf{P}_n(0)-\mathbf{P}_n(0)=\frac{1}{n}(\P^{n+1}(0)-\P(0)).$
Since the sequence  $\{\P^{n+1}(0)\}_{n\in\n}$ is bounded in $L^{\frac{n}{2},\infty}(\mathbb{R}^n)$,
we obtain that 
$$\lim_{n\to\infty}(\P\mathbf{P}_n(0)-\mathbf{P}_n(0))=\lim_{n\to\infty}\frac{1}{n}(\P^{n+1}(0)-\P(0))=0 \hbox{ strongly in } L^{\frac{n}{2},\infty}(\mathbb{R}^n).$$
Therefore, for the subsequence  $\{\mathbf{P}_{n_k}(0)\}$ from \eqref{suse} we have 
$\P\mathbf{P}_{n_k}(0)-\mathbf{P}_{n_k}(0)\ {\xrightarrow{weak\hbox{-}^*}}{}\ 0.$
This limit, together with \eqref{suse} implies that 
\begin{equation}\label{suse2}
\P\mathbf{P}_{n_k}(0)\ {\xrightarrow{weak\hbox{-}^*}}{}\ \hat{\xi}\in L^{\frac{n}{2},\infty}(\mathbb{R}^n).
\end{equation}
We now show that $\P(\hat{\xi})=\hat{\xi}$. To do this, we use the formula \eqref{Po1} and denote by $\tvh{\cdot,\cdot}$ the dual pair between $L^{\frac{n}{2},\infty}(\mathbb{R}^n)$ and $L^{\frac{n}{n-2},1}(\mathbb{R}^n)$. 
Then, since $e^{T\Delta}$ leaves $L^{\frac{n}{n-2},1}(\mathbb{R}^n)$ invariant for all $h\in L^{\frac{n}{n-2},1}(\mathbb{R}^n)$, we have
\begin{eqnarray}
\tvh{\P\mathbf{P}_{n_k}(0),h}
&=&\tvh{e^{T\Delta}\mathbf{P}_{n_k}(0),h} +
\tvh{\int_0^T \nabla_x \cdot e^{(T-s)\Delta }\left[ -\kappa\omega\nabla_x(-\Delta + \gamma I)^{-1}\omega + f\right](s)ds,h}\cr
&=&\tvh{\mathbf{P}_{n_k}(0),e^{T\Delta}h} + \tvh{\int_0^T \nabla_x \cdot e^{(T-s)\Delta }\left[ -\kappa\omega\nabla_x(-\Delta + \gamma I)^{-1}\omega + f\right](s)ds,h}\cr
&\xrightarrow{n_k\to\infty}&\tvh{\hat{\xi},e^{T\Delta}h} + \tvh{\int_0^T \nabla_x \cdot e^{(T-s)\Delta }\left[ -\kappa\omega\nabla_x(-\Delta + \gamma I)^{-1}\omega + f \right](s)ds,h}\cr
&=&\tvh{e^{T\Delta}\hat{\xi},h} +  \tvh{\int_0^T \nabla_x \cdot e^{(T-s)\Delta }\left[- \kappa\omega\nabla_x(-\Delta + \gamma I)^{-1}\omega + f\right](s)ds,h}\cr
&=&\tvh{\P(\hat{\xi}),h}.
\end{eqnarray} 
This yields that 
\begin{equation}\label{suse3}
\P\mathbf{P}_{n_k}(0)\ {\xrightarrow{weak\hbox{-}^*}}{}\  \P\hat{\xi}\in 
L^{\frac{n}{2},\infty}(\mathbb{R}^n).
\end{equation} 
It now follows from \eqref{suse2} and \eqref{suse3} that 
\begin{equation}\label{suse4}\P(\hat{\xi})=\hat{\xi}.
\end{equation}
Taking now the element $\hat{\xi}\in L^{\frac{n}{2},\infty}(\mathbb{R}^n)$ as an initial condition, 
by Assertion (i), there exists a unique mild solution $\hat{u}(\cdot)\in C_b(\r_+, L^{\frac{n}{2},\infty}(\mathbb{R}^n))$ satisfying $\hat{u}(0)=\hat{\xi}$. From the definition of Poincar\'e map $\P$ we arrive at $\hat{u}(0)=\hat{u}(T)$. Therefore,
the solution $\hat{u}(t)$ is periodic with the period $T$. The inequality \eqref{boundednessPeriodic} now follows from inequalities \eqref{boundedness} and \eqref{suse}.

{We now prove the uniqueness of the $T$-periodic mild solution. Let $\hat{u}_1$ and
$\hat{u}_2$ be two $T$-periodic mild solutions to system \eqref{LinearKSRn} which belong to  $ C_b(\r_+, L^{\frac{n}{2},\infty}(\mathbb{R}^n)$. Then, putting
$v=\hat{u}_1-\hat{u}_2$ we have that $v$ is $T$-periodic and by the integral formula \eqref{inte2},
\begin{equation}\label{Stokes2}
v(t)=e^{t\Delta }(\hat{u}_1(0)-\hat{u}_2(0))\hbox{ for }t>0.
 \end{equation}
By utilizing the $L^{r/2,\infty}-L^{n/2,\infty}$ dispersive estimates (see Assertion (i) in Lemma \ref{YamaEs}), we  have for $t>0$ and $r>n$:
\begin{eqnarray*}
\norm{v(t)}_{L^{\frac{r}{2},\infty}} &=& \norm{e^{t\Delta }(\hat{u}_1(0)-\hat{u}_2(0))}_{L^{\frac{r}{2},\infty}}\cr
&\leq& C t^{-\frac{n}{2}\left( \frac{2}{n} - \frac{2}{r} \right)} \norm{\hat{u}_1(0)-\hat{u}_2(0)}_{L^{\frac{n}{2},\infty}}.
\end{eqnarray*}
Combining this inequality with the fact that $t^{-\frac{n}{2}\left( \frac{2}{n} - \frac{2}{r} \right)}$ tends to zero as $t$ tends to infinity for $r>n$, we deduce that
\begin{equation}\label{Stab}
\lim_{t\to \infty}\|v(t)\|_{L^{\frac{r}{2},\infty}}=0.
 \end{equation} 
This fact, together with the periodicity of $v$, implies that $v(t)=0$ for all $t> 0$. Similar to \eqref{limittozero}, we use the weakly-$^*$ convergence \eqref{limit} and the heat semigroup $e^{t\Delta}$ is strongly continuous at $t=0$ on $L^{\frac{n}{2},1}(\mathbb{R}^n)$ to obtain that $\left<v(0),\phi\right>=0$ for all $\phi\in L^{\frac{n}{2},1}(\mathbb{R}^n)$. Therefore, we have $\hat{u}_1(t)=\hat{u}_2(t)$ for all $t\geqslant 0$.} The proof of Assertion (ii) is completed.
\end{proof}
By extending method in \cite{HuyXuan2016}, we state and prove the polynomial stability of periodic mild solution of system \eqref{KSRn} in the following theorem.
\begin{theorem}\label{main2}
If $f\in C_b(\r_+, L^{\frac{n}{3},\infty}(\mathbb{R}^n))$ is $T$-periodic function with the small enough norm, then {the following assertions hold}
\begin{itemize}
\item[(i)] {There exists a function $\hat{u}_0 \in L^{\frac{n}{2},\infty}(\mathbb{R}^n)$ which guarantees that the system \eqref{KSRn} with this initial data has a unique $T$-periodic mild solution $\hat{u}$} in a small ball of $C_b(\r_+, L^{\frac{n}{2},\infty}(\mathbb{R}^n))$. 

\item[(ii)] The above solution $\hat{u}$ is polynomial stable in the sense that: for another bounded mild solution $u(t)$ of system \eqref{KSRn} with initial data $u(0)$, if $\norm{u(0) - \hat{u}(0)}_{L^{\frac{n}{2},\infty}}$ is small enough, then
\begin{equation}\label{stability}
\norm{\hat{u}(t) - u(t)}_{L^{r,\infty}} \leqslant D t^{-\left( 1-\frac{n}{2r} \right)}
\end{equation}
for all $t>0$ and $r>\frac{n}{2}$.
\end{itemize}
\end{theorem}
\begin{proof}
$(i)$ First, we prove the existence of $\hat{u}$ based on fixed point arguments. 
 We consider the following closed ball $\B_\rho^{T}$ defined by
\begin{equation}\label{bro}
\B_\rho^{T}:=\{\omega\in C_b(\r_+, L^{\frac{n}{2},\infty}(\mathbb{R}^n): \omega \hbox{ is $T$-periodic and }\|\omega\|_{\infty,L^{\frac{n}{2},\infty}} \le \rho \}.
\end{equation}
{For $\omega \in  \B_\rho^{T}$, by using Theorem \ref{main1} there exists a unique initial data $u_0$ such that with this initial data the inhomogeneous linear system \eqref{LinearKSRn} has a unique $T$-periodic solution $u\in C_b(\mathbb{R}_+,L^{\frac{n}{2},\infty}(\mathbb{R}^n))$.
This $T$-periodic mild solution satisfies the following integral equation
\begin{equation}\label{ns1}
u(t) = e^{t\Delta}u_0 + \int_0^t \nabla_x \cdot e^{(t-s)\Delta}\left[ -\kappa \omega\nabla_x(-\Delta + \gamma I)^{-1}\omega + f\right](s)ds
\end{equation}
in the dual sense in $L^{\frac{n}{2},\infty}(\mathbb{R}^n)$- setting and the convergence
\begin{equation}\label{limit1}
\lim_{t\to 0^+} \left< u(t),\phi\right> = \left< u_0,\phi\right>, \hbox{   for all   } \phi \in L^{\frac{n}{2},1}(\mathbb{R}^n).
\end{equation}
Therefore, we can define the following transformation $\Phi: \mathcal{B}_\rho^T \to \mathcal{B}^T_\rho$ as follows 
\begin{equation}\label{defphi}
\begin{split}
\Phi(\omega)&=u\hbox{ where 
$u\in C_b(\r_+, L^{\frac{n}{2},\infty}(\mathbb{R}^n))$ is the unique $T$-periodic mild solution}\cr
& \hskip 5cm\hbox{  satisfying integral equation \eqref{ns1} and } \eqref{limit1}.
\end{split}
\end{equation}
}
We will prove that if $\|\omega\|_{\infty,L^{\frac{n}{2},\infty}}$ is small enough, then the transformation $\Phi$ acts from  $\B_\rho^{T}$ into itself and is a contraction. Indeed, applying Assertion (ii) in Theorem \ref{main1}, we obtain that the $T$-periodic mild solution $u$ satisfying
\begin{eqnarray}\label{ephi}
\|u\|_{\infty,L^{\frac{n}{2},\infty}}&\le &(C+1) \left( \kappa K g(\gamma)\norm{\omega}^2_{\infty,L^{\frac{n}{2},\infty}}+ \widetilde{C}\norm{f}_{\infty,L^{\frac{n}{3}}(\mathbb{R}^n)}\right)\cr
&\le& \kappa (C+1)\left(K g(\gamma)\rho^2 + \widetilde{C}\norm{f}_{\infty,L^{\frac{n}{3}}(\mathbb{R}^n)}\right).
\end{eqnarray}
Therefore, if $\rho$ and $\norm{f}_{\infty,L^{\frac{n}{3}}(\mathbb{R}^n)}$ are small enough, then the map  $\Phi$ acts from  $\B_\rho^{T}$ into itself. Then, by integral equation \eqref{ns1}, we have the following representation of $\Phi$
\begin{equation}\label{defphi1}
\Phi(\omega)(t)= u(t)= e^{t\Delta} u_0 + \int_0^t \nabla_x \cdot e^{(t-s)\Delta}\left[-\kappa \omega\nabla_x(-\Delta + \gamma I)^{-1}\omega + f\right](s)ds \hbox{   for   } t>0.
\end{equation}
Therefore, for $u_1, u_2 \in \B_\rho^T$ by the representation \eqref{defphi1} we obtain that
$u:=\Phi(u_1)-\Phi(u_2)$ is the unique $T$-periodic mild solution to the equation 
$$u_t = \Delta u - \left( \kappa\nabla_x \cdot \left[ (u_1-u_2)\nabla_x (-\Delta + \gamma I)^{-1}u_1\right] + \kappa\nabla_x \cdot \left[ u_2\nabla_x (-\Delta + \gamma I)^{-1}(u_1-u_2)\right] \right)$$
with the initial data $u(0)= u_1(0)-u_2(0)$.
Thus, again by Assertion (ii) in Theorem \ref{main1}, we arrive at
\begin{eqnarray}\label{ephi3}
\|\Phi(u_1)-\Phi(u_2)\|_{\infty,L^{\frac{n}{2},\infty}}&\le& \kappa K g(\gamma)(C+1)\left( \norm{u_1}_{\infty,L^{\frac{n}{2},\infty}} + \norm{u_2}_{\infty,L^{\frac{n}{2},\infty}} \right)\norm{u_1-u_2}_{\infty,L^{\frac{n}{2},\infty}}\cr
&\le& 2\kappa K g(\gamma)(C+1)\rho\norm{u_1-u_2}_{\infty,L^{\frac{n}{2},\infty}}.
\end{eqnarray}
We hence obtain that if $\rho$ is sufficiently small, then $\Phi : \B_\rho^T\to \B_\rho^T$ is a contraction. Therefore, for this value of  $\rho$, there exists a unique fixed point $\hat{u}$ of $\Phi$. Combining this with the limit \eqref{limit1},  we obtain that $\hat{u}$ is the unique $T$-periodic mild solution to system \eqref{KSRn} (with the initial data $\hat{u}_0$) in the closed ball $\mathcal{B}_\rho^T$ for $\rho$ small enough.\\

\noindent
$(ii)$ Now, we prove the polynomial stability of $\hat{u}$ by using again fixed point arguments and smoothing estimates \eqref{smoothEs} and Yamazaki's estimates \eqref{YaEs}. 
Setting $v=u-\hat u$ we obtain that $v$ satisfies the equation
\begin{equation}\label{2.26}
v(t)=e^{t\Delta}(u(0)- \hat{u}(0)) - \kappa B(u,u)(t) + \kappa B(\hat{u},\hat{u})(t).
\end{equation}
For the $r>\frac{n}{2}$, we set 
$$ \mathbb M=\{v\in C_b({\mathbb R}_+,L^{\frac{n}{2},\infty}(\mathbb{R}^n)): \sup_{t\in\r_+} t^{1-\frac{n}{2r}}{\norm {v(t)}}_{L^{r,\infty}} < \infty \}$$
endowed with the norm ${\norm v}_{\mathbb M}={\norm v}_{\infty,L^{\frac{n}{2},\infty}} + \sup_{t\in\r_+}t^{1-\frac{n}{2r}}{\norm {v(t)}}_{L^{r,\infty}}$. 

Now, we prove that if $\norm{u(0) - \hat{u}(0)}_{L^{\frac{n}{2},\infty}}$ and ${\norm u}_{\infty,L^{\frac{n}{2},\infty}}$ are small enough then Equation $(\ref{2.26})$ has a unique solution in a small ball of $\mathbb M$.

Indeed, for $v\in \mathbb M$ we consider the mapping $\Phi$ defined formally by
$$\Phi (v)(t):=e^{t\Delta}(u(0) - \hat{u}(0)) - \kappa B(v+\hat u, v+\hat u)(t) + \kappa B(\hat u,\hat u)(t).$$
Let $B_\rho$ be a ball in $\mathbb M$ of radius $\rho$ and centered at $0$. We then prove that if $\norm{u(0)-\hat{u}(0)}_{\frac{n}{2},\infty}$ and $\rho$ are small enough, the transformation $\Phi$ acts from $B_\rho$ to itself and is a contraction. 
For $v\in \mathbb M$, arguing similarly as in the proof of Theorem \ref{main1} we obtain $\Phi (v)\in C_b({\mathbb R}_+, L^{\frac{n}{2},\infty}(\mathbb{R}^n))$. Furthermore
$$t^{1-\frac{n}{2r}}\Phi(v)(t)=t^{1-\frac{n}{2r}}e^{t\Delta}(u(0)-\hat{u}(0))+t^{1-\frac{n}{2r}}\kappa\left(   -B(v+\hat u, v+\hat u)(t)+ B(\hat u,\hat u)(t)\right).$$
By estimates \eqref{smoothEs} for the heat semigroup $(e^{t\Delta})_{t\geqslant 0}$, we obtain that
\begin{equation}\label{term1}
t^{1-\frac{n}{2r}}\norm{e^{t\Delta}(u(0)-\hat{u}(0))}_{L^{r,\infty}}\le C\norm{u(0)-\hat{u}(0)}_{L^{\frac{n}{2},\infty}}.
\end{equation}
We now estimates $t^{1-\frac{n}{2r}}\kappa\norm{\left( -B(v+\hat u, v+\hat u)(t)+ B(\hat u,\hat u)(t)\right)}_{L^{r,\infty}}$. For each $\phi \in L^{\frac{r}{r-1},1}(\mathbb{R}^n)$, we have
\begin{eqnarray}\label{2.27}
&&\left| \left<  -B(v+\hat u, v+\hat u)(t)+ B(\hat u,\hat u)(t), \phi \right>  \right|\cr
&=&\left|\left< \int_0^t \nabla_x \cdot e^{(t-s)\Delta}\left[ -(v+\hat{u})\nabla_x(-\Delta + \gamma I)^{-1}(v+\hat{u}) + \hat{u}\nabla_x(-\Delta + \gamma I)^{-1}\hat{u}\right](s)ds,\phi \right>\right|\cr
&=&\left| \int_0^t \left< \nabla_x \cdot e^{(t-s)\Delta}\left[ -(v+\hat{u})\nabla_x(-\Delta + \gamma I)^{-1}(v+\hat{u}) + \hat{u}\nabla_x(-\Delta + \gamma I)^{-1}\hat{u}\right](s),\phi  \right>  ds\right| \cr
&\le& \int_0^t \left|\left< \nabla_x \cdot e^{s\Delta}\left[ -(v+\hat{u})\nabla_x(-\Delta + \gamma I)^{-1}(v+\hat{u}) + \hat{u}\nabla_x(-\Delta + \gamma I)^{-1}\hat{u}\right](t-s),\phi  \right> \right|  ds\cr
&=&\int_{0}^{t/2} \left|\left< \left[ -(v+\hat{u})\nabla_x(-\Delta + \gamma I)^{-1}(v+\hat{u}) + \hat{u}\nabla_x(-\Delta + \gamma I)^{-1}\hat{u}\right](t-s), \nabla_x e^{s\Delta}\phi  \right> \right|  ds \cr
&&+\int_{t/2}^{t} \left|\left< \left[ -(v+\hat{u})\nabla_x(-\Delta + \gamma I)^{-1}(v+\hat{u}) + \hat{u}\nabla_x(-\Delta + \gamma I)^{-1}\hat{u}\right](t-s), \nabla_x e^{s\Delta}\phi  \right> \right|  ds. 
\end{eqnarray}
We then estimate the two integrals on the last line of $(\ref{2.27})$. 
Using Lemma \ref{invertEs}, estimates \eqref{smoothEs} and \eqref{YaEs} for the first term we have
\begin{eqnarray}\label{2.28}
&& \int_{0}^{t/2} \left|\left< \left[ -(v+\hat{u})\nabla_x(-\Delta + \gamma I)^{-1}(v+\hat{u}) + \hat{u}\nabla_x(-\Delta + \gamma I)^{-1}\hat{u}\right](t-s), \nabla_x e^{s\Delta}\phi  \right> \right|ds \cr
&\le&\int_{0}^{t/2}\norm {\left[ -(v+\hat{u})\nabla_x(-\Delta + \gamma I)^{-1}(v+\hat{u}) + \hat{u}\nabla_x(-\Delta + \gamma I)^{-1}\hat{u}\right](t-s)}_{L^{\frac{nr}{n+r},\infty}}\norm{\nabla_x e^{s\Delta}\phi}_{L^{\frac{nr}{nr-n-r},1}} ds\cr
&\le& \int_{0}^{t/2}g(\gamma)\left({\norm {v(t-\xi)+{\hat u}(t-\xi)}}_{L^{\frac{n}{2},\infty}}+{\norm {{\hat u}(t-\xi)}}_{L^{\frac{n}{2},\infty}} \right){\norm {v(t-\xi)}}_{L^{r,\infty}}{\norm {\nabla_x e^{s\Delta}\phi}}_{L^{\frac{nr}{nr-n-r},1}}  ds \cr
&\le&\big (\frac{t}{2}\big )^{-1+\frac{n}{2r}}g(\gamma)\left({\norm v}_{\mathbb M}+2{\norm {\hat u}}_{\infty,\frac{n}{2},\infty} \right){\norm v}_{\mathbb M} \int_{0}^{t/2}{\norm {\nabla_x e^{\Delta s}\phi}}_{L^{\frac{nr}{nr-n-r},1}}  ds\cr
&\le& C_1\big (\frac{t}{2}\big )^{-1+\frac{n}{2r}}\left({\norm v}_{\mathbb M}+2\norm{\hat u}_{\infty,L^{\frac{n}{2},\infty}} \right){\norm v}_{\mathbb M}{\norm \phi}_{L^{\frac{r}{r-1},1}}.
\end{eqnarray}
We next estimate the second term in the last line of $(\ref{2.27})$. In fact, using again Lemma \ref{invertEs} and smoothing estimates \eqref{smoothEs}, we have
\begin{eqnarray}\label{2.29}
&&\int_{t/2}^{t} \left|\left< \left[ -(v+\hat{u})\nabla_x(-\Delta + \gamma I)^{-1}(v+\hat{u}) + \hat{u}\nabla_x(-\Delta + \gamma I)^{-1}\hat{u}\right](t-s), \nabla_x e^{s\Delta}\phi  \right> \right|ds \cr
&\le&\int_{t/2}^{t} \norm {\left[ -(v+\hat{u})\nabla_x(-\Delta + \gamma I)^{-1}(v+\hat{u}) + \hat{u}\nabla_x(-\Delta + \gamma I)^{-1}\hat{u}\right](t-s)}_{L^{\frac{n}{3},\infty}}\norm{\nabla_x e^{s\Delta}\phi}_{L^{\frac{n}{n-3},1}} ds\cr
&\le& \int_{t/2}^{t}g(\gamma)\left({\norm {v(t-\xi)+{\hat u}(t-\xi)}}_{L^{\frac{n}{2},\infty}}+{\norm {{\hat u}(t-\xi)}}_{L^{\frac{n}{2},\infty}} \right){\norm {v(t-\xi)}}_{\frac{n}{2},\infty}{\norm {\nabla_x e^{\Delta s}\phi}}_{L^{\frac{n}{n-3},1}}  ds \cr
&\le& C_2g(\gamma)\big (\frac{t}{2}\big )^{-1 + \frac{n}{2r}}\left({\norm v}_{\mathbb M}+2{\norm {\hat u}}_{\infty,L^{\frac{n}{2},\infty}} \right){\norm v}_{\mathbb M} \int_{t/2}^{\infty}s^{-2+\frac{n}{2r}}{\norm \phi}_{L^{\frac{r}{r-1},1}}  ds\cr
&\le& C_2g(\gamma)\big (\frac{t}{2}\big )^{-1 + \frac{n}{2r}}\left({\norm v}_{\mathbb M}+2{\norm {\hat u}}_{\infty,L^{\frac{n}{2},\infty}}\right){\norm v}_{\mathbb M}{\norm \phi}_{L^{\frac{r}{r-1},1}}.
\end{eqnarray}
Combining now $(\ref{2.27}), (\ref{2.28})$ and $(\ref{2.29})$ we obtain
\begin{eqnarray}\label{term2}
&&t^{1-\frac{n}{2r}}\kappa\norm{\left( -B(v+\hat u, v+\hat u)(t)+ B(\hat u,\hat u)(t)\right)}_{L^{r,\infty}}\cr
&\leqslant& 2^{1-\frac{n}{2r}}(C_1+C_2)g(\gamma)\left( {\norm v}_{\mathbb M}+2{\norm {\hat u}}_{\infty,L^{\frac{n}{2},\infty}} \right){\norm v}_{\mathbb M}{\norm \phi}_{L^{\frac{r}{r-1},1}}
\end{eqnarray}
for all $t>0$. Combining inequalities \eqref{term1} and \eqref{term2}, we get
$${\norm {\Phi(v)}}_{\mathbb M}\leqslant C \norm {u(0)-\hat{u}(0)}_{L^{\frac{n}{2},\infty}}+  2^{1-\frac{n}{2r}}(C_1+C_2)g(\gamma)\kappa\left( {\norm v}_{\mathbb M}+2{\norm {\hat u}}_{\infty,L^{\frac{n}{2},\infty}} \right){\norm v}_{\mathbb M}.$$
Similar calculations yield 
$${\norm {\Phi(v_1)-\Phi(v_2)}}_{\mathbb M}\le  2^{1-\frac{n}{2r}}(C_1+C_2)g(\gamma)\kappa\left( {\norm {v_1}}_{\mathbb M} + {\norm {v_2}}_{\mathbb M} +4{\norm {\hat u}}_{\infty,L^{\frac{n}{2},\infty}} \right){\norm {v_1-v_2}}_{\mathbb M}.$$
Therefore, if $ {\norm {u(0)-\hat{u}(0)}}_{L^{\frac{n}{2},\infty}}, {\norm {\hat u}}_{\infty,L^{\frac{n}{2},\infty}}$ and $\rho$ are small enough, then the transformation $\Phi$ acts from $B_\rho$ into itself and is contraction. As the fixed point of $\Phi$, the function $v=u-\hat u$ belongs to $\mathbb M$. The inequality \eqref{stability} hence follows, and we obtain the stability of the small solution $\hat u$.
\end{proof}

\section{Periodic solution for Keller-Segel (P-E) system on $\mathbb{H}^n$}\label{S3}
In this section, we extend to study the Keller-Segel (P-E) system on the real hyperbolic manifold $\mathbb{H}^n:=\mathbb{H}^n(\mathbb{R})$. Since the heat kernel (hence the heat semigroup) on the hyperbolic space $\mathbb{H}^n$ is exponential stable in $\mathbb{H}^n$, hence the bounded mild solution of the Keller-Segel (P-E) system on $\mathbb{H}^n$ will be exponential stable as a direct consequence.

\subsection{Keller-Segel (P-E) system on $\mathbb{H}^n$ and Pierfelice's dispersive estimates}
Let $(\mathbb{H}^{n},g)=(\mathbb{H}^{n}(\mathbb{R}),g)$ stand for a real hyperbolic manifold, where $n\geqslant 2$ is the dimension, endowed with a Riemannian metric $g$. This space is realized via a hyperboloid in $\mathbb{R}^{n+1}$ by considering the upper sheet
\[
\left\{  (x_{0},x_{1},...,x_{n})\in\mathbb{R}^{n+1};\text{ }\,x_{0}\geq1\text{
and }x_{0}^{2}-x_{1}^{2}-x_{2}^{2}...-x_{n}^{2}=1\,\right\}
\]
where the metric is given by $dg=-dx_{0}^{2}+dx_{1}^{2}+...+dx_{n}^{2}.$

In geodesic polar coordinates, the hyperbolic manifold $(\mathbb{H}^{n},g)$ can be described as
\[
\mathbb{H}^{n}=\left\{  (\cosh\tau,\omega\sinh\tau),\,\tau\geq0,\omega
\in\mathbb{S}^{n-1}\right\}
\]
with $dg=d\tau^{2}+(\sinh\tau)^{2}d\omega^{2},$ where $d\omega^{2}$ is the
canonical metric on the sphere $\mathbb{S}^{n-1}$. In these coordinates, the
Laplace-Beltrami operator $\Delta_{\mathbb{H}}$ on $\mathbb{H}^{n}$ can be
expressed as
\[
\Delta_{\mathbb{H}}=\partial_{r}^{2}+(n-1)\coth r\partial
_{r}+\sinh^{-2}r\Delta_{\mathbb{S}^{n-1}}.
\]
It is well known that the spectrum of $-\Delta_{\mathbb{H}}$ is the half-line
$\left[\dfrac{(n-1)^2}{4},\infty \right)$. 

We consider the Keller-Segel (P-E) system on whole real hyperbolic space $(\mathbb{H}^n,g)$, where $n \geqslant 2$: 
\begin{equation}\label{KSHn} 
\left\{
  \begin{array}{rll}
u_t \!\! &= \Delta_{\mathbb{H}} u + \nabla_x \cdot \left[ - \kappa u\nabla_x (-\Delta+ \gamma I)^{-1}u + f\right] \quad  &x\in  \mathbb{H}^n,\, t>0, \hfill \cr
u(0,x) \!\!& = \;u_0(x) \quad & x \in \mathbb{H}^n,
\end{array}\right.
\end{equation}
where $\Delta_{\mathbb{H}}$ is the Laplace-Beltrami operator on $(\mathbb{H}^n,g)$. As well as the case of Euclidean space $\mathbb{R}^n$, by Duhamel’s principle, the Cauchy problem \eqref{KSRn} can be formally converted to the following integral equation
\begin{equation}
u(t) = e^{\Delta_{\mathbb{H}} t}u_0 + \int_0^t \nabla_x \cdot e^{(t-s)\Delta_{\mathbb{H}}}\left[ - \kappa u\nabla_x(-\Delta_{\mathbb{H}} + \gamma I)^{-1}u + f\right](s)ds.
\end{equation}

The dispersive and smoothing estimates of heat kernel and heat semigroup on hyperbolic space are well studied in the literature for hyperbolic spaces and various types of noncompact manifolds for both the Laplace-Beltrami and the Hodge Laplacian  (see \cite{Anker1,Auscher,Bakry,Co,Da,Gri,Pi,Va, Vaz} and many others). Here, we will use the estimates of Pierfelice.
\begin{lemma}(\cite[Theorem 4.1 and corollary 4.3]{Pi})\label{estimates}
\begin{itemize}
\item[(i)] For $t>0$, and $p$, $q$ such that $1\leq p \leq q \leq \infty$, 
the following dispersive estimate holds: 
\begin{equation}\label{dispersive}
\left\| e^{t\Delta_{\mathbb{H}}}u_0\right\|_{L^q(\mathbb{H}^n)} \leq [h_n(t)]^{\frac{1}{p}-\frac{1}{q}}e^{-t( \gamma_{p,q})}\left\|u_0\right\|_{L^p(\mathbb{H}^n)} 
\end{equation}
for all $u_0 \in L^p(\mathbb{H}^n)$, where 
 $$h_n(t) = \hat{C}\max\left( \frac{1}{t^{n/2}},1 \right),\, 
   \gamma_{p,q}=\frac{\delta_n}{2}\left[ \left(\frac{1}{p} - \frac{1}{q} \right) + \frac{8}{q}\left( 1 - \frac{1}{p} \right) \right]$$ 
and $\delta_n$ is a positive constant depending only on $n$.  
\item[(ii)] For $t>0$, and $p,q$ such that $1\leqslant p\leqslant q \leqslant\infty$, the following estimate holds:
\begin{equation}
\left\| \nabla_x e^{t\Delta_{\mathbb{H}}} V_0 \right\|_{L^q(\mathbb{H}^n)} \leqslant [h_n(t)]^{\frac{1}{p}-\frac{1}{q}+\frac{1}{n}}e^{-t\left( \frac{\gamma_{q,q}+\gamma_{p,q}}{2} \right)} \left\|V_0\right\|_{L^p(\mathbb{H}^n)}
\end{equation}
for all vector field $V_0 \in L^p(\mathbb{H}^n)$. The functions $h_n(t)$ and $\gamma_{p,q}$ are defined as in Assertion (i).
\end{itemize}
\end{lemma}
Using the estimates in Lemma \ref{estimates}, Pierfelice and Maheux proved the well-posedness of local and global mild solutions of system \eqref{KSHn} on the two-dimensional hyperbolic space $\mathbb{H}^2$ (see \cite{MaPi}). Below, we will prove the well-posedness of global mild (periodic) solutions for all dimensions $n\geqslant 2$.

\subsection{Bounded periodic solutions and exponentially asymptotic behaviour}
By using Pierfelice's estimates in Lemma \ref{estimates}, the strategy for the existence of bounded periodic mild solution of system \eqref{KSHn} is similar to the case of system on Euclidean space $\mathbb{R}^n$. However, we need not to use Lorentz space to establish the boundedness of solution because we have a gain from the exponential decay rates of the dispersive and smoothing estimates of the heat semigroup $e^{t\Delta_{\mathbb{H}}}$ on the hyperbolic space.

The inhomogeneous linear system corresponding to system \eqref{KSHn} is
\begin{equation}\label{LinearKSHn} 
\left\{
  \begin{array}{rll}
u_t \!\! &= \Delta_{\mathbb{H}} u + \nabla_{x} \cdot \left[ -\kappa \omega\nabla_x (-\Delta_{\mathbb{H}} + \gamma I)^{-1}\omega + f\right] \quad  &x\in  \mathbb{H}^n,\, t>0, \hfill \cr
u(0,x) \!\!& = \;u_0(x) \quad & x \in \mathbb{H}^n
\end{array}\right.
\end{equation}
for a given function $\omega \in L^\infty(\mathbb{R}_+, L^{\frac{n}{2}}(\mathbb{H}^n))$.
This system can be reformulated by the following integral equation
\begin{equation}\label{inte22}
u(t) = e^{t\Delta_{\mathbb{H}}}u_0 + \int_0^t \nabla_x \cdot e^{(t-s)\Delta_{\mathbb{H}}}\left[-\kappa \omega\nabla_x(-\Delta_{\mathbb{H}} + \gamma I)^{-1}\omega + f\right](s)ds.
\end{equation}
The existence of bounded periodic mild solutions of the inhomogeneous linear system \eqref{LinearKSHn} and the Keller-Segel (P-E) system \eqref{KSHn} are established in the following theorem.
\begin{theorem}\label{main3} Let $n\geq 2$ and $n<p<2n$. Then, we have the following assertions:
\begin{itemize}
\item[(i)] (Boundedness) For given functions $\omega \in C_b(\mathbb{R}_+, L^{\frac{p}{2}}(\mathbb{H}^n))$ and $f\in C_b(\r_+, L^{\frac{p}{3}}(\mathbb{H}^n))$, there exists a unique mild solution of linear system \eqref{LinearKSHn} satisfying the integral equation \eqref{inte22}. Moreover, the following boundedness holds
\begin{equation}\label{boundedness1}
\norm{u(t)}_{L^{\frac{p}{2}}(\mathbb{H}^n)} \leq \hat C \norm{u_0}_{L^{\frac{p}{2}}(\mathbb{H}^n)} + \tilde{K}\left( \kappa g(\gamma)\norm{\omega}^2_{\infty,L^{\frac{p}{2}}(\mathbb{H}^n)} + \norm{f}_{\infty,L^{\frac{p}{3}}(\mathbb{H}^n)} \right).
\end{equation}
\item[(ii)] (Massera-type principle) If $\omega$ and $f$ are $T$-periodic functions depending on the time variable, {then there exists a function $\hat{u}_0 \in L^{\frac{p}{2}}(\mathbb{H}^n)$ which guarantees that the linear system \eqref{LinearKSHn} with the initial data $\hat{u}_0$ has a unique $T$-periodic mild solution $\hat{u}(t)$} satisfying the following boundedness
\begin{equation}\label{boundednessPeriodic1}
\norm{\hat{u}(t)}_{L^{\frac{p}{2}}(\mathbb{H}^n)} \leqslant (\hat{C}+1)\tilde{K}\left( \kappa g(\gamma)\norm{\omega}^2_{\infty,L^{\frac{p}{2}}(\mathbb{H}^n)} + \norm{f}_{\infty,L^{\frac{p}{3}}(\mathbb{H}^n)}\right).
\end{equation}
\item[(iii)] (Global well-posedness) If $f$ is a $T$-periodic function with small enough norm in $C_b(\r_+,L^{\frac{p}{3}}(\mathbb{H}^n))$, then {there exists a function $\hat{u}_0 \in L^{\frac{p}{2}}(\mathbb{H}^n)$ which guarantees that the system \eqref{KSHn} has a unique $T$-  periodic mild solution $\hat{u}$} in a small ball of $C_b(\r_+, L^{\frac{p}{2}}(\mathbb{H}^n))$.
\item[(iv)] (Exponential asymptotic behaviour) If the function $f$ satisfies that $\sup\limits_{t>0} e^{\sigma t}\norm{f(t)}_{L^{\frac{p}{3},\infty}(\mathbb{H}^n)}<+\infty$, then the above periodic mild solution $\hat{u}$ of {system \eqref{KSHn}} is exponential time-decay, i.e., it satisfies
\begin{equation}\label{exponentialdecay}
\norm{\hat{u}(t)}_{L^{\frac{p}{2}}(\mathbb{H}^n)} \leqslant De^{-\sigma t}
\end{equation}
for all $t>0$, where 
$\sigma= \min\left\{ \gamma_{p/2,p/2},\,\frac{\gamma_{p/2,p/2}+\gamma_{pn/(4n-p),p/2}}{2},\, \frac{\gamma_{p/2,p/2}+\gamma_{p/3,p/2}}{2}\right\}$.
\end{itemize}
\end{theorem}
\begin{proof}
We prove assertions (i) and (iv), {the proofs of assertions (ii) and (iii) are similar the ones of Assertion (ii) in Theorem \ref{main1} and Assertion (i) in Theorem \ref{main2}, respectively} for the space $L^{\frac{n}{2},\infty}(\mathbb{R}^n)$ is replaced by $L^{\frac{p}{2}}(\mathbb{H}^n)$.

(i) By using Lemma \ref{estimates} and the boundedness of $L_j = \partial_j(-\Delta + \gamma I)^{-1}$ (similary Lemma \ref{invertEs} and see also \cite[Lemma 3.3]{MaPi}), we have
\begin{eqnarray*}
\norm{u(t)}_{L^{\frac{p}{2}}(\mathbb{H}^n)} &\leq&\hat C \norm{e^{t\Delta_{\mathbb{H}}}u_0}_{L^{\frac{p}{2}}(\mathbb{H}^n)} + \kappa \int_0^t \norm{\nabla_x \cdot e^{(t-s)\Delta_{\mathbb{H}}}\left[ \omega\nabla_x(-\Delta_{\mathbb{H}} + \gamma I)^{-1}\omega \right](s)}_{L^{\frac{p}{2}}(\mathbb{H}^n)}ds\cr
&&+ \int_0^t \norm{\nabla_x \cdot e^{(t-s)\Delta_{\mathbb{H}}}f(s)}_{L^{\frac{p}{2}}(\mathbb{H}^n)}ds\cr
&\leq&\hat C e^{-t\gamma_{p/2,p/2}}\norm{u_0}_{L^{\frac{p}{2}}(\mathbb{H}^n)}\cr 
&&+ \kappa \int_0^t [h_n(t-s)]^{\frac{2}{p}}e^{-(t-s)\beta}\norm{\left[ \omega\nabla_x(-\Delta_{\mathbb{H}} + \gamma I)^{-1}\omega \right](s)}_{L^{\frac{pn}{4n-p}}(\mathbb{H}^n)}ds\cr
&&+ \int_0^t [h_n(t-s)]^{\frac{1}{p}+\frac{1}{n}}e^{-(t-s)\hat\beta}\norm{f(s)}_{L^{\frac{p}{3}}(\mathbb{H}^n)}ds\cr
&\leq&\hat C e^{-t\gamma_{p/2,p/2}}\norm{u_0}_{L^{\frac{p}{2}}(\mathbb{H}^n)}\cr
 &&+ \kappa \int_0^t [h_n(t-s)]^{\frac{2}{p}}e^{-(t-s)\beta}\norm{\omega(s)}_{L^{\frac{p}{2}}(\mathbb{H}^n)}  \norm{\left[\nabla_x(-\Delta_{\mathbb{H}} + \gamma I)^{-1}\omega \right](s)}_{L^{\frac{pn}{2n-p}}(\mathbb{H}^n)}ds\cr
&&+ \int_0^t [h_n(t-s)]^{\frac{1}{p}+\frac{1}{n}}e^{-(t-s)\hat\beta}\norm{f(s)}_{L^{\frac{p}{3}}(\mathbb{H}^n)}ds\cr
&\leqslant& \hat C \norm{u_0}_{L^{\frac{p}{2}}(\mathbb{H}^n)} + \int_0^t \left( (t-s)^{-\frac{n}{p}} +1 \right) e^{-(t-s)\beta}ds \left( \kappa \tilde C^{\frac{2}{p}} C g(\gamma)\norm{\omega}^2_{\infty, L^{\frac{p}{2}}(\mathbb{H}^n))} \right)\cr
&&+ \int_0^t \left( (t-s)^{-\frac{1}{2}-\frac{n}{2p}} +1 \right) e^{-(t-s)\hat\beta}ds \left(\tilde C^{\frac{1}{p}+\frac{1}{n}} \norm{f}_{\infty,L^{\frac{p}{3}}(\mathbb{H}^n)}\right)\cr
&\leq&\hat C \norm{u_0}_{L^{\frac{p}{2}}(\mathbb{H}^n)} + \left( \frac{1}{\beta^{1- \frac{n}{p}}}\Gamma\left(1- \frac{n}{p} \right) + \frac{1}{\beta} \right)\left(\kappa \tilde C^{\frac{2}{p}} C g(\gamma)  \norm{\omega}^2_{\infty, L^{\frac{p}{2}}(\mathbb{H}^n)} \right)\cr
&&+  \left(\frac{1}{\hat \beta^{\frac{1}{2} - \frac{n}{2p}}}\Gamma\left( \frac{1}{2} - \frac{n}{2p} \right) + \frac{1}{\hat \beta} \right) \tilde C^{\frac{1}{p}+\frac{1}{n}} \norm{f}_{\infty,L^{\frac{p}{3}}(\mathbb{H}^n)} \cr
&\leqslant&\hat C \norm{u_0}_{L^{\frac{p}{2}}(\mathbb{H}^n)} + \tilde{K}\left(\kappa g(\gamma)   \norm{\omega}^2_{L^\infty(\r_+, L^{\frac{p}{2}}(\mathbb{H}^n))} + \norm{f}_{\infty,L^{\frac{p}{3}}(\mathbb{H}^n)}\right),
\end{eqnarray*}
 where  $\beta=\dfrac{\gamma_{p/2,p/2}+\gamma_{pn/(4n-p),p/2}}{2}$,\; $\hat \beta=\dfrac{\gamma_{p/2,p/2}+\gamma_{p/3,p/2}}{2}$ and $$\tilde{K}=\max\left\lbrace  \left( \frac{1}{\beta^{1- \frac{n}{p}}}\Gamma\left(1- \frac{n}{p} \right) + \frac{1}{\beta} \right)  \tilde C^{\frac{2}{p}} C, \left( \frac{1}{\hat\beta^{\frac{1}{2} - \frac{n}{2p}}}\Gamma\left( \frac{1}{2} -\frac{n}{2p} \right) + \frac{1}{\hat\beta}\right) \tilde C^{\frac{1}{p}+\frac{1}{n}}\right\rbrace.$$ The proof of Assertion (i) is completed.

(iv) For a given periodic mild solution $\hat{u}$ in the small ball $B_\rho^T$ of $C_b(\r_+,L^{\frac{p}{2}}(\mathbb{H}^n))$, i.e., $\norm{\hat{u}}_{C_b(\r_+,L^{\frac{p}{2}}(\mathbb{H}^n)}<\rho$, we have
\begin{eqnarray}\label{Gron}
&&\norm{\hat{u}(t)}_{L^{\frac{p}{2}}(\mathbb{H}^n)} \leqslant \norm{e^{t\Delta_{\mathbb{H}}}u_0}_{L^{\frac{p}{2}}(\mathbb{H}^n)} + \kappa \int_0^t \norm{\nabla_x \cdot e^{(t-s)\Delta_{\mathbb{H}}}\left[ \hat{u}\nabla_x(-\Delta_{\mathbb{H}} + \gamma I)^{-1} \hat{u} \right](s)}_{L^{\frac{p}{2}}(\mathbb{H}^n)}ds\cr
&&+\int_0^t \norm{\nabla_x \cdot e^{(t-s)\Delta_{\mathbb{H}}}f(s)}_{L^{\frac{p}{2}}(\mathbb{H}^n)}ds\cr
&\leq&\hat C e^{-t\gamma_{p/2,p/2}}\norm{u_0}_{L^{\frac{p}{2}}(\mathbb{H}^n)}  + \kappa \int_0^t [h_n(t-s)]^{\frac{2}{p}}e^{-(t-s)\beta}\norm{\left[ \hat u\nabla_x(-\Delta_{\mathbb{H}} + \gamma I)^{-1}\hat u \right](s)}_{L^{\frac{pn}{4n-p}}(\mathbb{H}^n)}ds\cr
&&+ \int_0^t [h_n(t-s)]^{\frac{1}{p}+\frac{1}{n}}e^{-(t-s)\hat\beta}\norm{f(s)}_{L^{\frac{p}{3}}(\mathbb{H}^n)}ds\cr
&\leq&\hat C e^{-t\gamma_{p/2,p/2}}\norm{u_0}_{L^{\frac{p}{2}}(\mathbb{H}^n)}  + \kappa \int_0^t [h_n(t-s)]^{\frac{2}{p}}e^{-(t-s)\beta}\norm{\hat u(s)}_{L^{\frac{p}{2}}(\mathbb{H}^n)}  \norm{\left[\nabla_x(-\Delta_{\mathbb{H}} + \gamma I)^{-1}\hat u \right](s)}_{L^{\frac{pn}{2n-p}}(\mathbb{H}^n)}ds\cr
&&+ \int_0^t [h_n(t-s)]^{\frac{1}{p}+\frac{1}{n}}e^{-(t-s)\hat\beta}\norm{f(s)}_{L^{\frac{p}{3}}(\mathbb{H}^n)}ds\cr
&\leqslant& \hat C e^{-t\gamma_{p/2,p/2}}\norm{u_0}_{L^{\frac{p}{2}}(\mathbb{H}^n)} +\kappa \tilde C^{\frac{2}{p}} C g(\gamma) \int_0^t \left( (t-s)^{-\frac{n}{p}} +1 \right) e^{-(t-s)\beta} \norm{\hat u}_{ L^{\frac{p}{2}}(\mathbb{H}^n))}ds \norm{\hat u}_{C_b(\r_+, L^{\frac{p}{2}}(\mathbb{H}^n))}\cr
&&+ \int_0^t [h_n(t-s)]^{\frac{1}{p}+\frac{1}{n}}e^{-(t-s)\hat\beta}\norm{f(s)}_{L^{\frac{p}{3}}(\mathbb{H}^n)}ds\cr
&\leq& \hat C e^{-t\gamma_{p/2,p/2}}\norm{u_0}_{L^{\frac{p}{2}}(\mathbb{H}^n)} +\kappa \tilde C^{\frac{2}{p}} C g(\gamma) \rho \int_0^t \left( (t-s)^{-\frac{n}{p}} +1 \right) e^{-(t-s)\beta} \norm{\hat u}_{ L^{\frac{p}{2}}(\mathbb{H}^n))}ds  \cr
&&+ \int_0^t [h_n(t-s)]^{\frac{1}{p}+\frac{1}{n}}e^{-(t-s)\hat\beta}\norm{f(s)}_{L^{\frac{p}{3}}(\mathbb{H}^n)}ds.
\end{eqnarray}
Setting $y(t) = e^{\sigma t}\norm{\hat{u}(t)}_{L^{\frac{p}{2}}(\mathbb{H}^n)}$, where $\sigma= \min\left\{ \gamma_{p/2,p/2},\, \beta,\, \hat \beta \right\}$. The inequality \eqref{Gron} leads to
\begin{eqnarray}\label{Gron1}
y(t) \leq \hat C\norm{u_0}_{L^{\frac{p}{2}}(\mathbb{H}^n)} &+& \int_0^t [h_n(t-s)]^{\frac{1}{p}+\frac{1}{n}}e^{-(t-s)(\hat \beta-\sigma)}ds\sup_{t>0} e^{\sigma t} \norm{f(t)}_{L^{\frac{p}{3}}(\mathbb{H}^n)} \cr
&+&  \kappa \tilde C^{\frac{2}{p}} C g(\gamma) \rho  \int_0^t \left( (t-s)^{-\frac{n}{p}} +1 \right) e^{-(t-s)(\beta-\sigma)} y(s)ds.
\end{eqnarray}
Since the boundedness of integrals 
$$\int_0^t [h_n(t-s)]^{\frac{1}{p}+\frac{1}{n}}e^{-(t-s)(\hat \beta-\sigma)}ds<  \frac{\tilde C^{\frac{1}{p}+\frac{1}{n}}}{(\hat \beta-\sigma)^{\frac{1}{2} - \frac{n}{2p}}}{\bf\Gamma}\left( \frac{1}{2} - \frac{n}{2p} \right) + \frac{\tilde C^{\frac{1}{p}+\frac{1}{n}}}{\hat\beta-\sigma}=\hat {D}<+\infty,$$
and 
$$ \int_0^t \left( (t-s)^{-\frac{n}{p}} +1 \right) e^{-(t-s)(\beta-\sigma)}ds<\left( \frac{1}{(\beta-\sigma)^{1- \frac{n}{p}}}\Gamma\left(1- \frac{n}{p} \right) + \frac{1}{\beta-\sigma} \right)=\tilde D<+\infty,   $$
utilizing Gronwall's inequality, we obtain from inequality \eqref{Gron1} that
\begin{equation}
y(t)\leq \left(\hat C \norm{u_0}_{L^{\frac{p}{2}}(\mathbb{H}^n)}+ \hat{D}\sup_{t>0}e^{\sigma t}\norm{f(t)}_{L^{\frac{p}{3}}(\mathbb{H}^n)} \right) e^{\kappa \tilde C^{\frac{2}{p}} C g(\gamma) \rho  \tilde{D}}
\end{equation}
for all $t>0$. This yields the exponential decay \eqref{exponentialdecay}.
Our proof is completed.
\end{proof}
\begin{remark}
\begin{itemize}
\item[(i)] {The difference between the conditions on the dimension $n$  in Theorem \ref{main1} ($n\geqslant 4$ in Euclidean case $\mathbb{R}^n$) and Theorem \ref{main3} ($n\geqslant 2$ in hyperbolic case $\mathbb{H}^n$) is arised from the proofs of estimates for the bilinear operator $B(\cdot,\cdot)$ given by formula \eqref{BilinearOP}. In particular, for Euclidean space $\mathbb{R}^n$ the bilinear estimate (see Assertion (ii) in Proposition \ref{usefulEs}) requires the condition $n \geqslant 4$ (see the proof of Theorem 3.1 in \cite{Fe2021}), meanwhile the bilinear estimate in hyperbolic space $\mathbb{H}^n$ requires only $n \geqslant 2$ (see the proof of Assertion (i) in Theorem \ref{main3}).}
\item[(ii)] As a direct consequence of the exponential decay proven in Assertion (iv) in Theorem \ref{main3}, we have the exponential stability of periodic solution $\hat{u}$ as follows: for another bounded mild solution $u\in C_b(\r_+,L^{\frac{p}{2}}(\mathbb{H}^n))$ of system \eqref{KSHn} with initial data $u(0)=u_1$, if $\norm{u_0-u_1}_{L^{\frac{p}{2}}(\mathbb{H}^n)}$ is small enough, then the following inequality holds
\begin{equation}\label{exponentialstability}
\norm{\hat{u}(t)-u(t)}_{L^{\frac{p}{2}}(\mathbb{H}^n)} \leqslant De^{-\sigma t}\norm{u_0-u_1}_{L^{\frac{p}{2}}(\mathbb{H}^n)}
\end{equation}
for all $t>0$, where $\sigma= \min\left\{ \gamma_{p/2,p/2},\,\frac{\gamma_{p/2,p/2}+\gamma_{pn/(4n-p),p/2}}{2},\, \frac{\gamma_{p/2,p/2}+\gamma_{p/3,p/2}}{2}\right\}$.
\end{itemize}
\end{remark}

\subsection{Unconditional uniqueness}\label{Su33}
The unconditional uniqueness, i.e., the uniqueness of mild solution is established with a large space of initial data was established for Keller-Segel (P-E) system on Euclidean space $\mathbb{R}^n$ in a recent work \cite{Fe2021}. There are some related works about the unconditional uniqueness theorem for fluid dynamic equations such as Navier-Stokes, Boussinesq equations such as \cite{Fe2016,FeXu2023,Fe2023,Xuan2022}. 
In this subsection, we establish also a unconditional uniqueness theorem for mild solution of system \eqref{KSHn} on the hyperbolic space $\mathbb{H}^n$. Our method is developed from \cite{Fe2021,FeXu2023, Xuan2022}.
\begin{theorem}\label{main4}
Let $0<L\leqslant \infty$, $u$ and $v$ be two mild solutions of the system \eqref{KSHn} in the class $C_b([0,L); \, L^{\frac{p}{2}}(\mathbb{H}^n)) \,\,$ $ (n<p<2n)$ with the same initial data $u_0$. Then, $u(t)=v(t)$ in $L^{\frac{p}{2}}(\mathbb{H}^n)$ for all $t\in [0,\, L)$.
\end{theorem}
\begin{proof}
We start by showing the existence of a time $L_{1}\in(0,L)$ such that
$u(\cdot,t)=v(\cdot,t)$ for all
$t\in\lbrack0,L_{1})$. For this purpose, we denote
\begin{eqnarray*}
&&\omega=u-v,\,\omega_{1}=e^{t\Delta_{\mathbb{H}}}u_{0}-u,\, \omega_{2}=e^{t\Delta_{\mathbb{H}}}u_{0}-v.
\end{eqnarray*}
We can write
\begin{eqnarray*}
&& u\nabla_x(-\Delta_{\mathbb{H}} + \gamma I)^{-1}u - v\nabla_x(-\Delta_{\mathbb{H}} + \gamma I)^{-1}v \cr
&&= \omega\nabla_x(-\Delta_{\mathbb{H}} + \gamma I)^{-1}u + v\nabla_x(-\Delta_{\mathbb{H}} + \gamma I)^{-1}\omega\cr
&&= \omega\nabla_x(-\Delta_{\mathbb{H}} + \gamma I)^{-1} e^{t\Delta_{\mathbb{H}}}u_0 + e^{t\Delta_{\mathbb{H}}}u_0 \nabla_x(-\Delta_{\mathbb{H}} + \gamma I)^{-1}\omega \cr
&& - \omega\nabla_x(-\Delta_{\mathbb{H}} + \gamma I)^{-1} \omega_1 - \omega_2 \nabla_x(-\Delta_{\mathbb{H}} + \gamma I)^{-1}\omega.
\end{eqnarray*}
It follows that
\begin{eqnarray}
&&\norm{\omega(t)}_{L^{\frac{p}{2}}} \cr &&= \kappa \left\Vert - \int_0^t \nabla_x \cdot e^{(t-s)\Delta_{\mathbb{H}}}\left[ u\nabla_x(-\Delta_{\mathbb{H}} + \gamma I)^{-1}u \right](s)ds + \int_0^t \nabla_x \cdot e^{(t-s)\Delta_{\mathbb{H}}}\left[ v\nabla_x(-\Delta_{\mathbb{H}} + \gamma I)^{-1}v \right](s)ds \right\Vert _{L^{\frac{p}{2}}}\cr
&&\leq  \kappa \left\Vert \int_0^t \nabla_x\cdot e^{(t-s)\Delta_{\mathbb{H}}}\left[ \omega\nabla_x(-\Delta_{\mathbb{H}} + \gamma I)^{-1} \omega_1 +  \omega_2 \nabla_x(-\Delta_{\mathbb{H}} + \gamma I)^{-1}\omega \right] (s)ds \right\Vert _{L^{\frac{p}{2}}}\cr
&&+ \kappa \left\Vert \int_{0}^{t} \nabla_x \cdot e^{(t-s)\Delta_{\mathbb{H}}} \left[  \omega\nabla_x(-\Delta_{\mathbb{H}} + \gamma I)^{-1} e^{t\Delta_{\mathbb{H}}}u_0 + e^{t\Delta_{\mathbb{H}}}u_0 \nabla_x(-\Delta_{\mathbb{H}} + \gamma I)^{-1}\omega \right](s) ds\right\Vert _{L^{\frac{p}{2}}}\cr
&&= I_{1}(t)+I_{2}(t). \label{EST1}
\end{eqnarray}
Here, for the sake of simplicity we denote the norm $\norm{\cdot}_{L^{\frac{p}{2}}(\mathbb{H}^n)}$ by $\norm{\cdot}_{L^{\frac{p}{2}}}$.

By the same way as the proof of Assertion (i) in Theorem \ref{main3}, the term $I_1(t)$ can be estimated as
\begin{equation}\label{EST2}
I_{1}(t) \leqslant 2\kappa g(\gamma) \tilde{K}\sup_{0<t<L_1}\left\Vert \omega(\cdot,t)
\right\Vert _{L^{\frac{p}{2}}} \left(  \sup_{0<t<L_1}\left\Vert \omega_{1}(\cdot,t)
\right\Vert _{L^{\frac{p}{2}}} + \sup_{0<t<L_1}\left\Vert \omega_{2}(\cdot,t)
\right\Vert _{L^{\frac{p}{2}}} \right).
\end{equation}

Consider $L_1<1$, using dispersive and smoothing estimates in Lemma \ref{dispersive} and H\"older's inequality, we estimate the second term $I_2(t)$ as follows
\begin{eqnarray}\label{EST3}
I_{2}(t)  &\leq& \kappa\int_{0}^{t} [h_n(t-s)]^{\frac{2}{p}}e^{-(t-s)\beta} \norm{\left[  \omega\nabla_x(-\Delta_{\mathbb{H}} + \gamma I)^{-1} e^{t\Delta_{\mathbb{H}}}u_0 + e^{t\Delta_{\mathbb{H}}}u_0 \nabla_x(-\Delta_{\mathbb{H}} + \gamma I)^{-1}\omega \right](s)}_{L^{\frac{pn}{4n-p}}}ds\cr
&&\hbox{(where   } \beta = \frac{\gamma_{p/2,p/2}+\gamma_{pn/(4n-p),p/2}}{2})\cr
&\leq& 2\kappa Cg(\gamma)\int_{0}^{t} [h_n(t-s)]^{\frac{2}{p}} \norm{\omega(s)}_{L^{\frac{p}{2}}}
\norm{e^{t\Delta_{\mathbb{H}}}u_{0}}_{L^{\frac{p}{2}}}ds\cr
&\leq&2\kappa Cg(\gamma) \sup_{0<t<L_1}\left\Vert\omega(\cdot,t)\right\Vert_{L^{\frac{p}{2}}} \left\Vert e^{t\Delta_{\mathbb{H}}}u_{0}\right\Vert _{L^{\frac{p}{2}}} \int_{0}^{t} [h_n(t-s)]^{\frac{2}{p}} ds \cr
&\leqslant& 2\kappa Cg(\gamma) \tilde{C}^{\frac{2}{p}}\sup_{0<t<L_1}\left\Vert\omega(\cdot,t)\right\Vert _{L^{\frac{p}{2}}} \left\Vert e^{t\Delta_{\mathbb{H}}}u_{0}\right\Vert _{L^{\frac{p}{2}}}\int_{0}^{t} s^{-\frac{n}{p}} ds  \cr
&&\hbox{(because   } [h_n(s)] = s^{-\frac{n}{2}} \hbox{   for   } 0<s<t<L_1<1)\cr
&\leqslant&\frac{2p}{p-n}\kappa Cg(\gamma) \tilde{C}^{\frac{2}{p}} \sup_{0<t<L_1}\left\Vert\omega(\cdot,t)\right\Vert _{L^{\frac{p}{2}}} \sup_{0<t<L_1}t^{\frac{p-n}{p}} \left\Vert e^{t\Delta_{\mathbb{H}}}u_{0}\right\Vert _{L^{\frac{p}{2}}}.
\end{eqnarray}

Inserting \eqref{EST2} and \eqref{EST3} into \eqref{EST1}, we obtain
\begin{equation}
\sup_{0<t<L_1}\left\Vert\omega(\cdot,t)
\right\Vert _{L^{\frac{p}{2}}}\leqslant A(L_1)\sup_{0<t<L_1}{
\left\Vert\omega (\cdot,t)\right\Vert _{L^{\frac{p}{2}}}},
\end{equation}
where 
\begin{eqnarray*}
A(L_1) &=& 2\kappa g(\gamma) \tilde{K} \left(  \sup_{0<t<L_1}\left\Vert \omega_{1}(\cdot,t)
\right\Vert _{L^{\frac{p}{2}}} + \sup_{0<t<L_1}\left\Vert \omega_{2}(\cdot,t)
\right\Vert _{L^{\frac{p}{2}}} \right)\cr
&&+\frac{2p}{p-n}\kappa Cg(\gamma) \tilde{C}^{\frac{2}{p}}\sup_{0<t<L_1} t^{\frac{p-n}{p}} \left\Vert e^{t\Delta_{\mathbb{H}}}u_{0}\right\Vert _{L^{\frac{p}{2}}}.
\end{eqnarray*}
{Using dispersive estimate \eqref{dispersive}, we can estimate
\begin{eqnarray*}
\limsup_{t\rightarrow0^{+}} t^{\frac{p-n}{p}}\left\Vert e^{t\Delta_{\mathbb{H}}}u_{0}\right\Vert _{L^{\frac{p}{2}}}
&\leq& \limsup_{t\to 0^+} t^{\frac{p-n}{p}} e^{-\gamma_{p/2,p/2}t} \left\Vert u_{0}\right\Vert _{L^\frac{p}{2}}\cr
&\leq& \limsup_{t\to 0^+} t^{\frac{p-n}{p}} \left\Vert u_{0}\right\Vert _{L^\frac{p}{2}}\cr
&\longrightarrow& 0,\hbox{  as  } t\to 0^+,
\end{eqnarray*}
because $t^{\frac{p-n}{p}} \to 0$ as $t\to 0^+$ for $p>n$.} Moreover, we have
\begin{equation}
\limsup_{t\rightarrow0^{+}}\left(  \norm{\omega_{1}(\cdot,t)}_{L^{\frac{p}{2}}}+ \norm{\omega_2
(\cdot,t)}_{L^{\frac{p}{2}}}\right)  =0
\end{equation}
since $e^{t\Delta_{\mathbb{H}}}u_{0}$, $u$ and $v$ all converge to $u_{0}$ in $L^{\frac{p}{2}}(\mathbb{H}^n)$ as $t\rightarrow0^{+}$. Therefore, we can choose $L_1>0$ small enough such that $A(L_1)<1$. As a direct consequence, we obtain $\omega(\cdot,t)=0$ for all $t\in\lbrack0,L_{1})$.

Now, we define $L_1^{\ast}$ as the supremum of all $\widetilde{L}\in(0,L)$ such that
\begin{equation}
u(\cdot,t)=v(\cdot,t)\text{ in } L^{\frac{p}{2}}(\mathbb{H}^n).
\end{equation}
If $L_1^{\ast}<L$, then $u(\cdot,t)=v(\cdot,t)$, for all $t\in\lbrack0,L_1^{\ast}).$ By the time-continuity of solutions, it follows that $u(\cdot,L_1^{\ast})=v(\cdot,L_1^{\ast})$. Thus, proceeding similarly as above, we obtain $u(\cdot,t)=v(\cdot,t)$, for all $t\in\lbrack L_1^{\ast},\, L_1^{\ast}+\sigma)$ and for some small $\sigma>0$. It follows that $u(\cdot,t)=v(\cdot,t)$ in $[0,L_1^{\ast}+\sigma)$, which contradicts to the definition of $L_1^{\ast}$. Therefore, $L_1^{\ast}=L$ and we have the desired result.
\end{proof}

\section{Conclusion and further discussion}\label{S4}
In this paper, we have considered the existence, uniqueness of bounded periodic mild solution of parabolic-eliptic Keller-Segel system on whole spaces such as Euclidean space $\mathbb{R}^n\, (n\geqslant 4)$ and real hyperbolic space $\mathbb{H}^n\, (n\geqslant 2)$. Both two cases use the same strategy to obtain periodic mild solutions by combining dispersive and smoothing estimates with Massera-type principle. The work on $\mathbb{R}^n$ needs to use Lorentz spaces to establish the boundedness of mild solutions due to the heat semigroup is polynomial stable. The main difference between the results on $\mathbb{R}^n$ and the one on $\mathbb{H}^n$ is the asymptotic behaviour of mild solutions: we have proved the polynomial stability for mild solutions in $\mathbb{R}^n$ case and exponential decay (hence exponential stability) for mild solutions in $\mathbb{H}^n$ case. The work obtained in this paper provides a comparison about well-posedness and asymptotic behaviours of periodic mild solutions of Keller-Segel system in two difference context: the flat space such as Euclidean space $\mathbb{R}^n\, (n \geqslant 4)$ and the curved space with negative Ricci curvature such as $\mathbb{H}^n\, (n\geqslant 2)$. 

The asymptotic behaviour of periodic mild solutions of Keller-Segel (P-E) system obtained in $\mathbb{R}^n$ is proven in Theorem \ref{main2} by showing polynomial stability with time decay rate is $t^{-\left( 1-\frac{n}{2r}\right)}$, where $r>\frac{n}{2}$. Hence, the solutions decay fastly in comparing with the ones of Navier-Stokes equations, the time decay rate is $t^{-\left( \frac{1}{2}-\frac{n}{2r}\right)}$, where $r>n$ (see for example \cite{HuyXuan2016,Ya2000} and references therein). This difference comes from the structure of nonlinear part of Keller-Segel system which contains an inverse Laplace operator, meanwhile the ones of Navier-Stokes equations have not. This inverse operator helps us to improve the dimension of phase spaces as well as decay rates of mild solutions.

For the work on $\mathbb{H}^n$, we have proved the global well-posedness of mild solutions in the space $C_b(\r_+,L^{\frac{p}{2}}(\mathbb{H}^n))$, where $p>n$. The well-posedness in the case $1<p\leqslant n$ is an interesting question, where we will be faced the non-convergence of gamma functions. Our recent work \cite{XuanVan2023} can provide a useful tool to overcome this obstacle. Moreover, since the dispersive and smoothing estimates (known as Pierfelice's estimates) are still valid for the heat semigroup on generalized noncompact manifolds with negative Ricci curvatures which satisfy conditions $(H_1)-(H_4)$ in \cite{Pi}, our work can be extended to these manifolds.\\

\noindent
{\bf Acknowledgment.} P.T. Xuan was funded by the Postdoctoral Scholarship Programme of Vingroup Innovation Foundation (VINIF), code VINIF.2023.STS.55 and 
N.T.V. Anh was funded by Vietnam Ministry of Education and Training, under grant B2023-SPH-13.

\end{document}